\documentclass[a4paper,12pt]{article}
\usepackage{amsmath,amsfonts,amssymb,amsthm, amsrefs, amsrefs}

\newcommand{\Rset}{\mathbb{R}}
\newcommand{\Nset}{\mathbb{N}}

\newcommand{\Cset}{\mathbb{C}}
\newcommand{\PP}{\mathbb{P}}
\newcommand{\EE}{\mathbb{E}}
\newcommand{\DS}{\displaystyle} 

\newtheorem{Theorem}{Theorem}[section]

\newtheorem{Proposition}[Theorem]{Proposition}
\newtheorem{Lemma}[Theorem]{Lemma}
\newtheorem{Corollary}[Theorem]{Corollary}
\newtheorem{Remark}[Theorem]{Remark}

\title{The Kolmogorov operator associated to a Burgers SPDE in spaces of continuous functions}
\author{Luigi Manca\\Dipartimento di Matematica P. e A.,\\Universit\`a di Padova,\\Via Trieste 63,\\ 35121 Padova, Italy\\ \emph{E-mail: manca@math.unipd.it}}
\begin{document}

\maketitle

\begin{abstract}
We are concerned with a viscous Burgers equation forced by a perturbation of white noise type.
We study the corresponding transition semigroup 
 in a space of continuous functions weighted by a proper potential, and we show that the  infinitesimal generator  
 is the closure (with respect to a suitable topology) of the Kolmogorov operator associated to the stochastic equation.
In the last part of the paper we use this result to solve the  
corresponding Fokker-Planck equation.

{\bf Key words:} Burgers equation, white noise, Kolmogorov operator, transition semigroup, Fokker-Planck equation. 

{\bf MSC 2000:} 35Q53 (primary), 60H15, 35R15, 47D07 (secondary).
%
\end{abstract}


\section{Introduction and preliminaries} 

We consider  the  stochastic Burgers equation in the interval $[0,1]$  with Dirichlet boundary conditions
perturbed by a space-time white noise
\begin{equation}
\label{e.B.1}
     \left\{\begin{array}{l}
           \displaystyle{dX=\left(D^2_\xi X +\frac12\;
         D_\xi(X^{2})\right)dt       +      dW,\quad\xi \in [0,1],\;t\geq 0,}\\
          \\
      X(t,0)=X(t,1)=0  
\\
    \\
     X(0,\xi )=x(\xi ),\;\xi \in [0,1],
\end{array}\right.
\end{equation}
where $x\in  L^2(0,1)$ %
and  $W$ is a  cylindrical  Wiener process defined in a probability space $(\Omega ,{\cal F},\PP)$   and with   values in $L^2(0,1)$.\bigskip

The unknown $X$ is a real valued process depending on $\xi\in [0,1]$ and $t\geq0$.
It is known that there exists a unique solution with paths in $C([0,T],L^2(0,1))$ (see \cite{DPDT}).
In the paper \cite{DPD07} it has been proved that all the moments of the solution are finite.
Starting by this result, we study the associated transition semigroup and the associated Kolmogorov operator in spaces of continuous functions with polynomial growth.

Let us write problem \eqref{e.B.1} in an abstract form.
 We  denote by  $L^p(0,1),\;p\ge 1,$ 
the space of all real valued Lebesgue measurable functions %
 $x:[0,1]\to\Rset$ such that
\[
      |x|_{p}:=\left(\int_0^{1}|x(\xi)|^pd\xi  \right)^{1/p}<+\infty,
\]
and by $L^\infty(0,1)$ the space of all real valued Lebesgue measurable essentially bounded functions endowed with the usual supremum norm $|x|_\infty$, $x\in L^\infty(0,1)$.
We consider the separable Hilbert space $ H:=L^2(0,1)$ (norm $|\cdot|_2$ and inner product $\langle x,y\rangle=|xy|_1$, $x,y\in H$). 
As usual, $H^k(0,1)$, $k\in\Nset$, is the Sobolev space of all functions in $H$ whose differentials  belong to $H$ up to the order $k$, and $H^1_0(0,1)$ is the subspace of $H^1(0,1)$ of all functions whose trace at $0$ and $1$ vanishes. 
We define the unbounded self-adjoint operator $A$ in  $H$ by
\[
  Ax=\frac{\partial^2}{\partial\xi^2}x,\quad x\in D(A)=H^2(0,1)\cap H^1_0(0,1).
\]
Finally, we denote by $\{e_k\}_{k\in \Nset}$ the orthonormal system in $H$ given by the eigenvectors of $A$
\[
  e_k(\xi)=\sqrt{\frac{2}{\pi}}\sin(k\xi),\quad \xi \in [0,1],\, k\in \Nset.
\]
The cylindrical Wiener process $W(t)$  is given (formally) by
\[
    W(t)=\sum_{k=1}^\infty \beta_k(t)e_k,\quad t\ge 0,
\]
where $\{\beta_k\}$ is a  sequence of mutually independent standard Brownian
motions on a stochastic basis
$(\Omega,\mathcal F,(\mathcal F_t)_{t\geq0},\PP)$. 
We set
\[
      b(x)=\frac12\;D_\xi(x^{2}),\quad x\in D(b)= H^1_0(0,1).
\]
Thanks to the introduced notations, we write problem \eqref{e.B.1} as
\begin{equation}   \label{e.B.3}
 \left\{\begin{array}{lll}
     dX&=&(AX+b(X))dt+ dW(t),\\\\
     X(0)&=&x,\quad x\in H
 \end{array}\right. 
\end{equation}

The unique solution of  \eqref{e.B.3} is denoted by $X(t,x)$.
The corresponding transition semigroup $P_t$, $t\geq0$ is given by
\begin{equation} \label{e.B.38a}
 P_t\varphi(x)=\EE[\varphi(X(t,x))],\quad t\geq0,\,\varphi\in \mathcal C_b(H),\, x\in H,
\end{equation}
where $\mathcal C_b(H)$ is the Banach space of all continuous and bounded real valued functions $\varphi:H\to \Rset$ endowed with the supremum norm 
\[
   \|\varphi\|_0=\sup_{x\in H}|\varphi(x)|
\] 
and $\EE$ means expectation.
In \cite{DPD07} it has been proved the following 
\begin{Proposition}\label{p.B.7}
For any $p\geq2,\,k\geq1$,   $T>0$ there exists a constant $c_{ p,k,T}$ such that  
 \[
  \EE\left[\sup_{t\in [0,T]} |X(t,x)|_{p}^k   \right]\leq c_{p,k,T}(1+|x|_{p}^k).
\]
\end{Proposition}
Then the semigroup $P_t$ is extendible to spaces of real valued continuous functions with polynomial growth.
In particular, $P_t$ acts on the space 
$\mathcal C_{b,V}(L^6(0,1))$ of all  continuous functions  $\varphi:L^6(0,1)\to \Rset$ such that the function
\[
   L^6(0,1)\to\Rset,\quad x\mapsto \frac{ \varphi(x)}{1+V(x)} 
\]
is bounded and  
where 
\[
  V(x):=  |x|_{6}^8|x|_{4}^2,\quad x\in L^6(0,1).
\]
The space $\mathcal C_{b,V}(L^6(0,1))$, endowed with the norm 
\[
   \|\varphi\|_{0,V}:=\sup_{x\in L^6(0,1)}  \frac{|\varphi(x)|}{1+V(x)}  
\]
is a Banach space.

The semigroup $P_t$  is not strongly continuous in $\mathcal C_{b,V}(L^6(0,1))$ (neither in $\mathcal C_b(H)$).
However, it is strongly continuous with respect to weaker topologies.
We follow here the approach of the $\pi$-convergence, suggested by Priola in \cite{Priola}.
The semigroup $P_t$ can be also studied in other frameworks, for instance with respect to uniform convergence on compact sets (see \cite{Cerrai}, \cite{GK01} and Proposition \ref{p.B.3.4} below).
We define the infinitesimal generator of $P_t$ by setting 
\begin{equation}  \label{e.B.0}
\begin{cases}
\DS D(K,\mathcal C_{b,V}(L^6(0,1)))=\bigg\{ \varphi \in \mathcal C_{b,V}(L^6(0,1)): \exists g\in \mathcal C_{b,V}(L^6(0,1)),   \\ 
  \DS  \qquad  \lim_{t\to 0^+} \frac{ {P}_t\varphi(x)-\varphi(x)}{t}= g(x),\,x\in L^6(0,1),\;
    \sup_{t\in(0,1)}\left\|\frac{ {P}_t\varphi-\varphi}{t}\right\|_{0,V}<\infty \bigg\}\\    
   {}   \\
  \DS  {K}\varphi(x)=\lim_{t\to 0^+} \frac{ {P}_t\varphi(x)-\varphi(x)}{t},\quad \varphi\in D(K,\mathcal C_{b,V}(L^6(0,1))),\,x\in L^6(0,1).
\end{cases}
\end{equation} 

The Kolmogorov operator associated to equation \eqref{e.B.3} is formally given by
\begin{equation} \label{e.B.4}
 K_0\varphi(x)= \frac12\textrm{Tr}\big[D^2\varphi(x)\big]+\langle x,AD\varphi(x)\rangle-\frac12\langle D_\xi D\varphi(x),x^2\rangle,
\end{equation}
where $\varphi:L^p(0,1)\to \Rset$ is a suitable function and Tr means trace. 

The main result of this paper consists in Theorem \eqref{T.B.2} below where we show that $(K,D(K,\mathcal C_{b,V}(L^6(0,1))))$ is the closure, with respect to the $\pi$-convergence, of the operator $K_0$ define on the domain $\mathcal E_A(H)$ (the set of {\em exponential functions}), 
  which consists of the linear span of the real and imaginary part of the functions 
\[
  H\to \Cset,\quad x\mapsto e^{i\langle x,h\rangle}.
\quad h\in D(A). 
\]
In other words, we show the set of exponential functions $\mathcal E_A(H)$ is a core for $(K,D(K,\mathcal C_{b,V}(L^6(0,1))))$, and $K\varphi=K_0\varphi$, $\forall \varphi\in \mathcal E_A(H)$.

Apart for the interest in having a better understanding of the  operator $K_0$, the main motivation is to solve the corresponding Fokker-Planck equation
\[
   \begin{cases}
  \DS    \frac{d}{dt} \mu_t= K_0^*\mu_t,\quad t\geq0\\
      \mu_0=\mu,
   \end{cases}
\]
where the family of measures $\mu_t$, $t\geq0$ in the unknown and $\mu$ is a given Borel measure on $H$.
The meaning of this problem will be explained by Theorem \ref{T.B.1.4} below. 
Before introducing it, we need some notation.

%
If $E$ is a Banach space, $\mathcal M(E)$ is the set of all  Borel finite measures on $E$.
If $\mu\in \mathcal   M(E)$, we denote by $|\mu|_{TV}$ the total variation measure of $\mu$.
We shall denote by $\mathcal M_V(L^6(0,1))$ the set of all $\mu\in \mathcal M(L^6(0,1))$ such that
\[
   \int_{L^6(0,1)}V(x)|\mu|_{TV}(dx)<\infty.
\]
The second main result of this paper is 
\begin{Theorem} \label{T.B.1.4}
For any    $\mu \in \mathcal M_V(L^6(0,1))$ %
there exists an unique family of measures $\{\mu_t,\;t\geq0\}\subset \mathcal M_V(L^6(0,1))$  fulfilling 
\begin{equation} \label{e.B.5b}
\int_0^T\left(\int_{L^6(0,1)}(1+ V(x))|\mu_t|_{TV}(dx)\right)dt < \infty,\quad \forall T>0 
\end{equation}
and the Fokker-Planck equation 
\begin{equation}\label{e.B.2.5a}
\int_{L^6(0,1)}\varphi(x)\mu_t(dx)- \int_{L^6(0,1)}\varphi(x)\mu(dx)=
\int_0^t\left( \int_{L^6(0,1)}K_0\varphi(x)\mu_s(dx)   \right)ds,  
\end{equation}
$t\geq0,\,\varphi \in\mathcal E_A(H)$.
Moreover,  the  solution  is given by $P_t^*\mu$, ${t\geq0}$.
\end{Theorem}
The meaning of $P_t^*\mu$ will be made clear by Theorem \ref{T.B.1} below.
\bigskip

In the papers \cite{RS04}, \cite{RS06}, the stochastic partial differential equation
\begin{equation} \label{e.I2.1}
  \begin{cases}
      dX_t=\left(\Delta X(t)+ F(X(t))  \right)dt + \sqrt{Q}\, dW(t)\\
      X(0)=x\in H,
  \end{cases}
\end{equation}
 has been considered, where
 $H:=L^2(0,1)$, $W(t),\, t\geq0$ is a cylindrical Wiener process on $H $, 
$Q:H\to H$ is a nonnegative definite symmetric operator of trace class, $\Delta$ is the Dirichlet Laplacian on $(0,1)$, $F:H^1_0(0,1)\to H$ is a measurable vector field of the form 
\[
   F(x)(r)= \frac{d}{dr}\left(\Psi\circ x  \right)(r)+\Phi(r,x(r)),\quad x\in H^1_0(0,1),\, r\in (0,1).
\]
The associated Kolmogorov operator is
\[
  \mathcal L\varphi(x)=\frac12\textrm{Tr}\left(QD^2\varphi(x)  \right)+\left\langle\Delta x+F(x),D\varphi(x)\right\rangle,  
\] 
where $\varphi:H\to \Rset$ is a suitable cylindrical smooth function.
Roughly speaking, the authors show that $\mathcal L$ can  be extended to the generator of a strongly continuous semigroup in  a space of weakly continuous functions weighted by a proper Lyapunov-type function.
Then, they construct a Markov process which solves equation \eqref{e.I2.1} in the sense of the martingale problem.
 \bigskip

We stress that in \cite{RS04}, \cite{RS06} the noise is driven by a trace class operator, whereas in our case the perturbation is of white noise type.  
In this direction, the results of Theorem \ref{T.B.1} seem to be new.

We mention also  the papers \cite{GK01}, \cite{Manca07}, \cite{Manca07a}, where Kolmogorov operators of Ornstein-Uhlenbeck and reaction-diffusion type are considered in spaces of uniformly continuous functions.

In  \cite{DPD07} (see, also, \cite{DPZ92}, \cite{DPZ96}, \cite{DPZ02}, \cite{DP04}, \cite{DPT01}, \cite{Manca06}) the operator $K_0$ has been considered in the space $L^2(H;\nu)$, where $\nu$ is the invariant measure (its existence is proved in \cite{DPDT}) of the semigroup $P_t$, $t\geq0$.
In addition, several estimates are proved in order to ensure that the operator $K_0$ is $m$-dissipative in $L^2(H;\nu)$.
Therefore, the authors show that the operator $K_0$ can be uniquely extended to the infinitesimal generator of the semigroup $P_t$ in $L^2(H;\nu)$.

\bigskip

The approach we use has been developed in \cite{Manca07}, where a Fokker-Planck equation has been considered for Ornstein-Uhlenbeck operator perturbed by a Lipschitz and bounded term.
In the paper \cite{Manca07a}, this approach has been extended also to reaction-diffusion operators, and the Fokker-Planck equation has been solved for measure with finite moments up to a suitable degree.

We mention also \cite{LaTh}, where Markov transition semigroups on spaces of measures have been considered, and the theory of such semigroups is developed.

Existence of measure valued solutions for equations involving second order partial differential operators in infinite dimensional spaces has been also considered in \cite{BR01}.
However, in this paper we concentrate on uniqueness of the solution, whereas in \cite{BR01} it has been shown existence results.

The paper is organized as follows:
in the next two sections we shall introduce notation and we derive some preliminary results that will be used throghout the paper. In section 4, we shall study the transition semigroup $P_t$ in in $\mathcal C_{b,V}(L^6(0,1))$.
In section 5, we shall introduce the Ornstein-Uhlenbeck semigroup and the corresponding operator in the spaces $\mathcal C_{b,1}(H)$ and $\mathcal C_{b,V}(L^6(0,1))$.
Section 6 is devoted to study the Galerkin approximations of  problems \eqref{e.B.1}, \eqref{e.B.3} and to derive some fundamental estimate on the associated transition semigroup.
In section 7 we show the main result of the paper, that is that the Kolmogorov operator $K_0$ is the characterization on a core of the operator \eqref{e.B.0}. 
Finally, section 8 is devote to the proof of Theorem \ref{T.B.1}.
\section{Notations and preliminary results}

%
If $E$, $E'$ are, respectively, a topological space and a  Banach space with norm $|\cdot|_{E'}$, we denote by $\mathcal C_b(E,E')$ the Banach space of the 
bounded continuous function $\varphi:E\to E'$ endowed with the supremum norm 
\[
   \|\varphi\|_{\mathcal C_b(E,E')}:=\sup_{x\in E}|\varphi(x)|_{E'}.
\]
When $E'=\Rset$ we write $\mathcal C_b(E)=\mathcal C_b(E,\Rset)$.
If $E=H$, we simply denote by $\|\cdot\|_0$ the supremum norm of $\mathcal C_b(H)$.
We also denote by $\mathcal C_{b,1}(H)$ the Banach space of all continuous functions $f:H\to \Rset$ such that 
\[
  \|f\|_{0,1}:=\|(1+|\cdot|_2)^{-1}f\|_0<\infty.
\]
The set $\mathcal C_b^1(H)$ is the space of all $\varphi\in \mathcal C_b(H)$ which are Fr\'echet differentiable with continuous and bounded differential $D\varphi \in \mathcal C_b(H,H)$.

We now introduce the definition of $\pi$-convergence. 
Here, $E$ denote a topological space.
A sequence $\{\varphi_n\}_{n\in \Nset} \subset\mathcal C_b(E)$ is said to be {\em $\pi$-convergent} to a function $\varphi$ $\in$ $\mathcal C_b(E)$  if 
\[
 \lim_{n\to \infty}\varphi_n(x)=\varphi(x),\quad \forall x\in E
\] 
and 
\[
  \sup_{n\in\Nset}\|\varphi_n\|_{\mathcal C_b(E)} <\infty. 
\]
Similarly,  the $m$-indexed sequence $\{\varphi_{n_1,\ldots,n_m}\}_{n_1\in\Nset,\ldots,n_m\in\Nset}\subset \mathcal C_b(E)$ is said to be $\pi$-convergent to  $\varphi$ $\in$ $\mathcal C_b(E)$ if for any $i\in \{1,\ldots,m-1\}$  
 there exists an $i$-indexed sequence $\{\varphi_{n_1,\ldots,n_{i}}\}_{n_1\in\Nset,\ldots,n_{i}\in\Nset}\subset \mathcal C_b(E)$  such that 
\[
 \lim_{n_{i+1}\to\infty}\varphi_{n_1,\ldots,n_{i+1}}\stackrel{\pi}{=}
    \varphi_{n_1,\ldots,n_i},\quad  i\in \{1,\ldots,m-1\}
\]
and
\[
  \lim_{n_1\to\infty}\varphi_{n_1}\stackrel{\pi}{=}\varphi.
\]
We shall write
\[
 \lim_{n_1\to \infty}\cdots\lim_{n_m\to \infty} \varphi_{n_1,\ldots,n_m}\stackrel{\pi}{=}\varphi
\] 
or $\varphi_n\stackrel{\pi}{\to}\varphi$ as $n \to \infty$, when the sequence has one index.
%
%
%
%
%
%


We recall that the operator $A$  is the infinitesimal generator of a strongly continuous semigroup in $H$, which we denote by $e^{tA}$, $t\geq0$.
Moreover, 
the semigroup $e^{tA}$ can be extended to $L^p(0,1)$, for any $p\geq1$, and for all $p>1$ the constant $\lambda_p=2p^{-1}(p-1)\pi^2$ verifies
\[
   |e^{tA}x|_p\leq e^{\lambda_pt}|x|_p,\quad x\in L^p(0,1).
\]


\section{Estimates on the solution}
Here we collect some properties of the solution of \ref{e.B.1}.
A fundamental role will be played by the so-called stochastic convolution $W_A(t)$, which is formally given by
\begin{equation} \label{e.B.15}
   W_A(t)=\int_0^t e^{(t-a)A}dW(s)=\sum_{k=1}^\infty\int_0^t e^{(t-a)A}e_kd\beta_k(s).
\end{equation}
For any $t>0$, the process $W_A(t)$ has gaussian law of zero mean and covariance operator
\[
   Q_tx=\int_0^t e^{2(t-a)A}xds,\quad x\in H
\]
As easily seen, the operator $Q_t$ is trace class.
Now set
\[
   Y(t,x)=X(t,x)-W_A(t).
\]
We write \eqref{e.B.3} as
\begin{equation}   \label{e.B.10}
 \left\{\begin{array}{lll}
   \DS  Y(t,x)&=&\DS e^{tA}x + \int_0^t e^{(t-s)A}\frac{\partial}{\partial\xi}\left(Y(s,x)+W_A(s,x)  \right)^2 ds,\\ \\
   \DS  Y(0,x)&=&x,\quad x\in H
 \end{array}\right. 
\end{equation}
As we shall see, if $z(t)\in L^p(0,1)$ a.s., then $e^{tA}\frac{\partial}{\partial\xi}z^2\in L^p(0,1)$ is bounded.
Then the above integral converges and the equation is meaningful. 
We now give the precise meaning of solution.
We say that $X(t,x)$ is a {\em mild} solution of \eqref{e.B.3} if $Y(t,x)=X(t,x)-W_A(t)$ satisfies \eqref{e.B.10} for a.s. all $\omega\in\Omega$.
The following result is proved in \cite{DPDT}.
\begin{Theorem} \label{t.B.7.1}
Let $x\in L^p(0,1)$, $p\geq2$.
Then there exists a unique mild solution of equation \eqref{e.B.3}, which belongs $\PP$-a.s. to $C([0,T];L^p(0,1))$, for any $T>0$.
\end{Theorem}

We prove  uniform continuity with respect to the initial datum in a bounded neighborhood.
In order to proceed, set
\begin{equation}\label{e.B.12}
   \theta_{}=\sup_{t\in [0,T]}|W_A(t)|_\infty,\quad T>0.
\end{equation}
Clearly $\theta_{}$ is a random variable, and $\theta_{}<\infty$ a.s. 
We need the following estimates, proved in Lemma 3.1 of \cite{DPDT}
\begin{Lemma} \label{l.B.3.1}
 For any $p\in[2,\infty)$ there exists $c_p>0$ such that 
 if $Y(t,x)$ is a solution of \eqref{e.B.10}, then
\[
   |Y(t,x)|_p\leq c_p\left(\theta_{}^3+|x|_p   \right)e^{1+2p\theta_{} t}. 
\]
\end{Lemma}
We have the following
\begin{Theorem}  \label{t.B.3.2}
For any $p\in[2,\infty)$ there exists a continuous function $c_p:(\Rset^+)^4\to \Rset^+$ such that 
\[
   |Y(t,x)-Y(t,y)|_p\leq c_p(t,|x|_p,|y|_p,\theta_{})|x-y|_p, \quad  x,\,y\in L^p(0,1)
\]
\end{Theorem}
\begin{proof}
Here we follow \cite{DPDT}.
By \eqref{e.B.10}  we have
\begin{multline*}
   Y(t,x)-Y(t,y)= e^{tA}(x-y)\\
   +\frac12\int_0^t e^{(t-s)A}\frac{\partial}{\partial \xi}\big(( Y(s,x)-Y(s,y))(Y(s,x)+Y(s,y)+2W(s))  \big)ds.
\end{multline*}
then
\begin{multline} \label{e.B.11}
   |Y(t,x)-Y(t,y)|_p\leq |x-y|_p\\
   +\frac12\int_0^t\left| e^{(t-s)A}\frac{\partial}{\partial \xi}\big(( Y(s,x)-Y(s,y))(Y(s,x)+Y(s,y)+2W(s))  \big)\right|_pds.
\end{multline}
As well known, $e^{tA}$, $t\geq0$ has smoothing properties.
In particular, for any $s_1,s_2\in \Rset$, $s_1\leq s_2$, $r\geq1$, $e^{tA}$ maps $W^{s_1,r}(0,1)$ into $W^{s_2,r}(0,1)$, for any $t>0$.
Moreover,  there exists $C_1>0$, depending on $s_1,s_2,r$, such that
\begin{equation} \label{e.B.13}
   |e^{tA}z |_{W^{s_2,r}(0,1)}\leq C_1\left(1+t^{\frac{s_1-s_2}{2}}   \right)|z |_{W^{s_1,r}(0,1)},
   \quad z\in W^{s_1,r}(0,1),
\end{equation}
see Lemma 3, Part I in \cite{Rothe}.
Using the Sobolev embedding theorem we have
\[
  \left|e^{(t-s)A}\frac{\partial}{\partial \xi}\big(( Y(s,x)-Y(s,y))(Y(s,x)+Y(s,y)+2W(s))\big)\right|_p
\]
\[
 \leq   C_1 \left|e^{(t-s)A}\frac{\partial}{\partial \xi}\big( (Y(s,x)-Y(s,y))  
  (Y(s,x)+Y(s,y)+2W(s))\big)\right|_{W^{\frac1p,\frac p2}(0,1)}
\]
and, thanks to the above estimate with $s_1=-1$, $s_2=1/p$, $r=p/2$
\[
  \left|e^{(t-s)A}\frac{\partial}{\partial \xi}\big(( Y(s,x)-Y(s,y))(Y(s,x)+Y(s,y)+2W(s))\big)\right|_p
\]
\begin{multline*}
   \leq C_1C_2\left(1+(t-s)^{-\frac{1}{2}-\frac{1}{2p}}   \right) 
\\
  \times \left| \frac{\partial}{\partial \xi}\big(( Y(s,x)-Y(s,y))(Y(s,x)+Y(s,y)+2W(s))\big)\right|_{W^{-1,\frac p2}(0,1)}
\end{multline*}
\[
  \leq C_1C_2\left(1+(t-s)^{-\frac{1}{2}-\frac{1}{2p}}   \right) 
  \left|( Y(s,x)-Y(s,y))(Y(s,x)+Y(s,y)+2W(s))\right|_{\frac{p}{2}}
\]
\[
   \leq C_1C_2\left(1+(t-s)^{-\frac{1}{2}-\frac{1}{2p}}   \right) 
  \left|  Y(s,x)-Y(s,y)\right|_{p}  \left|Y(s,x)+Y(s,y)+2W(s)\right|_{p}
\]  
\begin{multline*}
\leq c_pC_1C_2\left(1+(t-s)^{-\frac{1}{2}-\frac{1}{2p}}   \right)
   \left|  Y(s,x)-Y(s,y)\right|_{p}\\
   \times  \left(\left(2\theta_{}^3+|x|_p+|y|_p \right)e^{1+2p\theta_{} s}+2 \right)
\end{multline*}
Now the result follows by \eqref{e.B.11} and by Gronwall lemma (see, for instance, Lemma 7.1.1 in \cite{Henry}).
\end{proof}
By recalling that $Y(t,x)=X(t,x)-W_A(t)$ it follows immediately the following result, which will be fundamental in the next section
\begin{Corollary} \label{c.B.8}
For any $p\in [2,\infty)$, $x\in L^p(0,1)$, $T>0$ we have, $\PP$-a.s.,
\[
    \lim_{\varepsilon \to 0^+} \sup_{|h|_p<\varepsilon}\left( \sup_{t\in [0,T]}|X(t,x+h)-X(t,x)|_p  \right)=0
\]
%
\end{Corollary}
\section{The transition semigroup in $\mathcal C_{b,V}(L^6(0,1))$}  

This section is devoted in studying the semigroup $  P_t$, ${t\geq0}$ in the  space\linebreak  $\mathcal C_{b,V}(L^6(0,1))$. 
\begin{Remark}\label{r.B.9}
{\em
By Corollary \ref{c.B.8} it follows that for any $\varphi\in \mathcal C_{b,V}(L^6(0,1))$
\[
 \lim_{\varepsilon\to 0}  \sup_{ |h|_6<\varepsilon,\,t\in[0,T]}|P_t\varphi(x+h)- P_t\varphi(x)|=0.
\]
This, together with the estimate of Theorem \ref{t.B.3.2},
allows us to show that $P_t$ maps the space $\mathcal C_{b,V}(L^6(0,1))$ into itself.  
}
\end{Remark}
\begin{Proposition} \label{p.B.2.1}
 Formula \eqref{e.B.38a}  defines a semigroup of operators $({P}_t)_{t\geq0} $ in $\mathcal C_{b,V}(L^6(0,1))$ 
and there exist  two constants $c_0\geq1$, $\omega_0\in \Rset$ and a family of probability measures 
 $\{\pi_t(x,\cdot),\,t\geq0,\,x\in L^6(0,1)\}\subset  \mathcal M_{V}(L^6(0,1))$ 
such that
\begin{itemize}
\item[(i)] $ {P}_t\in \mathcal L( \mathcal C_{b,V}(L^6(0,1)))$ and 
 $  \|  {P}_t\|_{\mathcal L(\mathcal C_{b,V}(L^6(0,1))} \leq  c_0e^{ \omega_0 \,t}$;  
\item[(ii)] $\DS  {P}_t\varphi(x)=\int_H\varphi(y)\pi_t(x,dy) $, for any $t\geq0$, 
$\varphi \in \mathcal C_{b,V}(L^6(0,1))$, $x\in L^6(0,1)$;  
\item[(iii)] for any  $\varphi \in \mathcal C_{b,V}(L^6(0,1))$, $x\in L^6(0,1)$, the function $\Rset^+\to\Rset$, $t\mapsto  {P}_t\varphi(x)$ is continuous.  
\item[(iv)] $  {P}_t  {P}_s=  {P}_{t+s}$, for any $t,s\geq0$ and $  {P}_0=I$;
\item[(v)] for any $\varphi\in \mathcal C_{b,V}(L^6(0,1))$ and any sequence
 $(\varphi_n)_{n\in \Nset}\subset \mathcal C_{b,V}(L^6(0,1))$ such that 
\[
  \lim_{n\to\infty}  \frac{\varphi_n}{1+V} \stackrel{\pi}{=}  \frac{\varphi}{1+V} 
\]
we have, for any $t\geq0$,  
\[
 \lim_{n\to\infty}   \frac{{P}_t\varphi_n}{1+V}  \stackrel{\pi}{=} \frac{{P}_t\varphi}{1+V}.
\]
\end{itemize}
\end{Proposition}
\begin{proof} 
 (i). Take   $\varphi\in   \mathcal C_{b,V}(L^6(0,1))$, $t>0 $. 
We have to show that $ {P}_t\varphi\in \mathcal C_{b,V}(L^6(0,1)) $.
By Proposition \ref{p.B.7} it follows that
\[
   \left|  {P}_t\varphi(x) \right|\leq \|\varphi\|_{0,V}(1+\EE[V(X(t,x))])\leq c\|\varphi\|_{0,V}(1+V(x)),
\]
for some $c>0$.
Then, we have to show that the function $L^6(0,1)\to\Rset$. $x\mapsto P_t\varphi(x) $ is continuous.
Fix $x_0\in L^6(0,1)$. %
We have
\[
   |P_t\varphi(x_0+h)-P_t\varphi(x_0)|\leq \EE[\varphi(X(t,x_0+h))-\varphi(X(t,x_0))|].
\]
By Corollary \ref{c.B.8} we have that $|X(t,x_0+h)-X(t,x_0)|_6\to 0$ 
$\PP$-a.s. as $|h|_6\to 0$.
Then, by the continuity of $\varphi$ it follows $|\varphi(X(t,x_0+h))-\varphi(X(t,x_0))|\to 0$ 
$\PP$-a.s. as $|h|_6\to 0$.
On the other hand,  $\varphi(X(t,x_0+h))$ has bounded expectation, uniformly in any $L^6(0,1)$-ball of center $x_0$.
Then, it follows that $P_t\varphi(x_0+h)\to P_t\varphi(x_0)$ has $|h|_6\to 0 $.
(i) is proved. \\
(ii). 
Take $\varphi\in \mathcal C_{b,V}(L^6(0,1))$, and consider a sequence $(\varphi_n)_{n\in\Nset}\subset \mathcal C_b(L^6(0,1))$ such that 
\begin{equation*} 
    \lim_{n\to\infty}\frac{\varphi_n}{1+V}\stackrel{\pi}{=}\frac{\varphi}{1+V}.
\end{equation*}
 Since $\pi_t(t,\cdot)$ is the image measure of $X(t,x)$ in $H$, the representation (ii) holds for any $\varphi_n$, that is
\[
   {P}_t\varphi_n(x)=\EE\bigl[\varphi_n(X(t,x))  \bigr]=\int_H\varphi_n(y)\pi_t(x,dy),\quad x\in H.
\]
By Proposition \ref{p.B.7} we can apply the dominated convergence theorem to find the result. 
 \\
(iii). For any $\varphi\in  \mathcal C_{b,V}(L^6(0,1))$, $x\in L^6(0,1)$, $t,s\geq0$ we have 
\begin{eqnarray*}
    {P}_t\varphi(x)- {P}_s\varphi(x)=\EE\left[\varphi(X(t,x))-\varphi(X(s,x))\right]
\end{eqnarray*}
Since for any $T>0$ we have $X(\cdot,x)\in C([0,T],L^6(0,1))$ $\PP$-a.s. (cfr. Theorem \ref{t.B.7.1}), (iii) follows.
\\
(iv).
Take $\varphi\in \mathcal C_{b,V}(L^6(0,1))$ and consider a sequence $(\varphi_n)_{n\in\Nset}\subset \mathcal C_b(L^6(0,1))$ such that $(1+V)^{-1} \varphi_n\stackrel{\pi}{\to}(1+V)^{-1} \varphi$ as $n\to\infty$.
By the markovianity of the process $X(t,x)$ it follows that (iv) holds true for any $\varphi_n$.
Then, since by (iii)  $(1+V)^{-1} {P}_t\varphi_n\stackrel{\pi}{\to}(1+V)^{-1} {P}_t\varphi$ as $n\to\infty$, still by (iii)  
we find
\[
 \frac{ {P}_{t+s}\varphi}{1+V }\stackrel{\pi}{=}  \lim_{n\to\infty}\frac{ {P}_{t+s}\varphi_n}{1+V }
  = \lim_{n\to\infty}\frac{ {P}_t {P}_s\varphi_n}{1+V }\stackrel{\pi}{=}\frac{ {P}_t {P}_s\varphi}{1+V}.
\]
(v). Since (ii) holds and $\pi_t(x,dy)\in\mathcal M_V(L^6(0,1))$, the result follows by the dominated convergence Theorem.
This concludes the proof.
\end{proof}
\begin{Proposition} \label{p.B.2.9}
Let   $X(t,x)$  be the mild solution of problem \eqref{e.B.3} and let $P_t$, ${t\geq0}  $   
be the associated transition semigroups in the   space  
$\mathcal C_{b,V}(L^6(0,1))$ defined by \eqref{e.B.38a}. 
Let also $(K,D(K,\mathcal C_{b,V}(L^6(0,1)))) $   
be the associated infinitesimal generators, defined by  \eqref{e.B.0}.  
Then
\begin{itemize}
\item[(i)] for any $\varphi \in D(K,\mathcal C_{b,V}(L^6(0,1)))$, we have $  {P}_t\varphi \in D(K,\mathcal C_{b,V}(L^6(0,1)))$ and 
     $  {K} {P}_t\varphi =  {P}_t  {K}\varphi$, $t\geq0$;
\item[(ii)] for any $\varphi \in D(K,\mathcal C_{b,V}(L^6(0,1)))$, $x\in H$, the map $[0,\infty)\to \Rset$, $t\mapsto  {P}_t\varphi(x)$ is continuously differentiable and $(d/dt) {P}_t\varphi(x) =  {P}_t  {K}\varphi(x)$;
\item[(iii)] for any $\varphi\in \mathcal C_{b,V}(L^6(0,1))$, $t>0$, the function 
\[
   H\to \Rset,\quad x\mapsto \int_0^tP_s\varphi(x)ds
\]
belongs to $D(K,\mathcal C_{b,V}(L^6(0,1)))$, and it holds 
\[
   K\left(  \int_0^tP_s\varphi ds \right)=P_t\varphi-\varphi;
\]
\item[(iv)]  %
 for any $\lambda> \omega_0$, where $\omega_0$ is as in Proposition \ref{p.B.2.1},  
 the linear operator $R(\lambda, {K})$ on $\mathcal C_{b,V}(L^6(0,1))$
defined by
\[
   R(\lambda,K)f(x)=\int_0^\infty e^{-\lambda t} {P}_tf(x)dt, \quad f\in \mathcal C_{b,V}(L^6(0,1)),\, x\in L^6(0,1) 
\]
satisfies, for any  $f\in \mathcal C_{b,V}(L^6(0,1))$ 
\[
R(\lambda,K)\in \mathcal L(\mathcal C_{b,V}(L^6(0,1))),\quad \quad \|R(\lambda,K)\|_{\mathcal L(\mathcal C_{b,V}(L^6(0,1)))}\leq 
\frac{ c_0 }{\lambda-\omega_0} 
\]
\[
  R(\lambda, {K})f\in D(K,\mathcal C_{b,V}(L^6(0,1))),\quad (\lambda I- {K})R(\lambda, {K})f=f,
\]
 where $c_0$ is as in Proposition \ref{p.B.2.1}.
We call $R(\lambda,K)$ the {\em resolvent} of $K$ at $\lambda$.
\end{itemize}
\end{Proposition}
\begin{proof}
(i), (ii) are an easy consequence of \eqref{e.B.0} and Proposition \ref{p.B.2.1} (see the proof of Proposition 2.8 in \cite{Manca07}). 

Let us show (iii).
First, we have to check that  $\int_0^t P_s f ds$ belongs to $\mathcal C_{b,V}(L^6(0,1))$.
By (i) of Proposition \ref{p.B.2.1}, for any $x\in L^6(0,1)$ we have 
\[
  \left|\int_0^t P_s\varphi(x)ds\right| \leq  \|\varphi\|_{0,V} c_0\int_0^te^{\omega_0 s}ds(1+V(x)).
\]
then,
\[
   \sup_{x\in L^6(0,1)} \frac{1}{1+V(x)}\left|\int_0^t P_s\varphi(x)ds\right|<\infty.
\]
Now let us fix $\varepsilon >0$, $x_0\in L^6(0,1)$ and take $\delta >0$ such that 
\[
  \sup_{s\in [0,t]} \sup_{\stackrel{h\in   L^6(0,1)}{|h|_6<\delta}}\left| P_s\varphi(x_0+h)  - P_s\varphi(x_0) \right| <\frac{\varepsilon}{t}.
\]
The constant $\delta>0$ exists thank to Remark \ref{r.B.9}.
Therefore, for any $h\in  L^6(0,1)$, $|h|_6<\delta$ we have
\[
  \left| \int_0^t   P_s\varphi(x_0+h)  ds -\int_0^t  P_s\varphi(x_0)   ds\right|\leq
  \int_0^t\left| P_s\varphi(x_0+h)  - P_s\varphi(x_0)   \right| ds <\varepsilon.
\]
By the arbitrariness of $x_0$, it follows  $\int_0^t    P_s\varphi    ds \in\mathcal C_{b,V}(L^6(0,1))$.
The rest of the proof is essentially the same done for Theorem 2.9 in \cite{Manca07}.
\end{proof}
\subsection{Further results}

In this section we show that the semigroup $P_t$ is strongly continuous with respect to the weak convergence which has been introduced in \cite{Cerrai}.
\begin{Proposition} \label{p.B.3.4}
For any $p\in [2,\infty)$ and any compact set $K\subset L^p(0,1)$  it holds
\[
   \lim_{t\to 0^+}\sup_{x\in K}\left|X(t,x)-x   \right|_p=0\quad \PP-\textrm{\em a.s}.
\]
\end{Proposition}
\begin{proof}
Fix $T>0$. 
As usual, $\theta_{}$ is the random variable in \eqref{e.B.12}.
For any $x\in L^p(0,1)$, $0<t<T$ we have
\[
  |X(t,x)-x|_p\leq |Y(t,x)-x|_p+|W_A(t)|_p
\]
\[
   \leq |e^{tA}x-x|_p+|W_A(t)|_p+\int_0^t\left| e^{(t-s)A}\frac{\partial}{\partial \xi}\big((Y(s,x)-W(s))^2\big)  \right|_p ds
\]
By arguing as in the proof of Theorem \ref{t.B.3.2} we find
\[
   \int_0^t\left| e^{(t-s)A}\frac{\partial}{\partial \xi}\big((Y(s,x)-W(s))^2\big)  \right|_p ds
\] 
\[
  \leq  C_1C_2\int_0^t \left(1+(t-s)^{-\frac{1}{2}-\frac{1}{2p}}   \right)\left|Y(s,x)-W(s))\right|_p ds
\]
for some $C_1, C_2>0$.
Thanks to Lemma \ref{l.B.3.1}, the last term in the right-hand side is bounded by
\[
   C_1C_2\int_0^t \left(1+(t-s)^{-\frac{1}{2}-\frac{1}{2p}}   \right)\left( c_p\left(\theta_{}^3+|x|_p   \right)e^{1+2p\theta_{} s}+\theta_{}\right)ds
\]
Finally, we have found 
\[
   |X(t,x)-x|_p\leq  |e^{tA}x-x|_p+|W_A(t)|_p
\]
\[
+C_1C_2\int_0^t \left(1+(t-s)^{-\frac{1}{2}-\frac{1}{2p}}   \right)\left( c_p\left(\theta_{}^3+|x|_p   \right)e^{1+2p\theta_{} s}+\theta_{}\right)ds.
\] 
According to the fact that $|e^{tA}x-x|_p\to 0$ as $t\to0^+$, uniformly on compact sets of $L^p(0,1)$, and that $|W(t)|_p\to0$ $\PP$-a.s. as $t\to 0^+$   we get the result.
\end{proof}

Thanks to Proposition \ref{p.B.3.4} we are able to improve the previous result.
We have
\begin{Proposition}  \label{p.B.4.3}
The semigroup $P_t$, $t\geq0$ is a strongly continuous semigroup in the mixed topology of $\mathcal C_{b,V}(L^6(0,1))$.
That is for any $\varphi\in \mathcal C_{b,V}(L^6(0,1))$, $T>0$ and any compact set $K\subset L^6(0,1)$ we have
\[
   \lim_{t\to 0^+}\sup_{x\in K}|P_t\varphi(x)-\varphi(x))| =0
\]
and
\[
   \sup_{t\in [0,T]}\|P_t\varphi\|_{0,V}<\infty.
\]
\end{Proposition}

\section{The Ornstein-Uhlenbeck operator} 
\label{s.B.OU}
Here we consider the transition semigroup 
associated to the mild solution of 
the linear stochastic equation
\begin{equation} \label{e.B.3.1}
\left\{\begin{array}{lll}
       dZ(t,x)&=&  AZ(t,x) dt+  dW(t),\quad t\geq 0\\
        \\
        Z(0,x)&=&x\in H.
\end{array}\right.
\end{equation}
We recall that the solution is given by the process 
\begin{equation}\label{e.B.4a}
    Z(t,x)=e^{tA}x+W_A(t),
\end{equation} 
where $W_A(t)$, $t\geq0$ is the stochastic convolution introduced in \eqref{e.B.15}.
We recall that the process $Z(t,x)$ has a version which is, a.s. for $\omega\in \Omega$, $\alpha$-H\"older continuous with respect to $(t,x)$, for any $\alpha\in (0,\frac14)$ (see \cite{DPZ92}, Theorem 5.20 and Example 5.21). %
For any $t\geq0$ 
we define the Ornstein-Uhlenbeck (OU) semigroup $R_t$, ${t\geq0}$ by 
\begin{equation} \label{e.B.23a}
  R_t\varphi(x)= \EE\left[\varphi(Z(t,x))  \right], \quad t\geq0,\,\varphi\in \mathcal C_b(H),\,x\in H
\end{equation}
where $Z(t,x)$ is the mild solution of \eqref{e.B.3.1}.
It is well known (see, for instance, \cite{DPD98}) the following result
\begin{Proposition} \label{p.B.15}
For any $p,k\geq1$, $T>0$ there exists a constant $c_{p,k,T}>0$ such that 
\begin{equation} \label{e.B.18}
   \EE\left[\sup_{t\in [0,T]}|Z(t,x)|_{p}^k\right]\leq c_{p,k,T}\left(1+|x|_{p}^k\right).
\end{equation}
\end{Proposition}
This easily implies that $R_t$, $t\geq0$ can be extended to a semigroup in the space $\mathcal C_{b,1}(H)$  (see also \cite{Cerrai}, \cite{GK01}).
In this space, $R_t$ is not strongly continuous with respect to the supremum norm.
However, it is well known that it is strongly continuous with respect to the so called mixed topology (see \cite{GK01}) or, equivalently (see \cite{GK01} for this result), with respect to the weakly convergence introduce by Cerrai in \cite{Cerrai}.
We recall that a sequence $\{\varphi_n\}_n\subset \mathcal C_{b,1}(H)$ is said to be weakly convergent to $\varphi\in \mathcal C_{b,1}(H)$ if $\varphi_n(x)\to \varphi(x)$, $\forall x\in H$, as $n\to \infty$ uniformly on compact sets of $H$ and 
$\sup_n\|\varphi_n\|_{0,1}<\infty$.
So, following \cite{Cerrai}, it is possible to define an infinitesimal operator 
\[
  L:D(L,\mathcal C_{b,1}(H))\subset \mathcal C_{b,1}(H)\to \mathcal C_{b,1}(H)
\]
through its resolvent operator (see Definition 3.3 and Remark 3.4 in  \cite{Cerrai}).

By Theorem 3.7 of \cite{Priola}, it follows that the operator $(L,D(L,\mathcal C_{b,1}(H)))$ coincides with the operator 
\begin{equation}  \label{e.OU.0}
\begin{cases}
\DS D(L,(L,\mathcal C_{b,1}(H)))=\bigg\{ \varphi \in \mathcal C_{b,1}(H): \exists g\in \mathcal C_{b,1}(H),    \\ 
  \DS  \qquad\lim_{t\to 0^+} \frac{ R_t\varphi(x)-\varphi(x)}{t}=g(x),\,\forall x\in H,\;\sup_{t\in(0,1)}\left\|\frac{R_t\varphi-\varphi}{t}\right\|_{0,1}<\infty \bigg\}\\   
   {}   \\
  \DS L\varphi(x)=\lim_{t\to 0^+} \frac{ R_t\varphi(x)-\varphi(x)}{t},\quad \varphi\in D(L,\mathcal C_{b,1}(H)),\,x\in H.
\end{cases}
\end{equation}

It is well known the following fact (see \cite{DPZ92}, \cite{DPZ02})
\begin{equation}\label{e.OU.35}
  R_te^{i\langle \cdot,h\rangle}(x)= e^{i\langle e^{tA}x,h\rangle-\frac12\langle Q_th,h\rangle}, \quad t\geq0,\, x,\,h\in H.
\end{equation}
This implies that
 $R_t:\mathcal E_A(H)\to \mathcal E_A(H)$, $\forall t\geq0$ or,
in other words, that the set $\mathcal E_A(H)$ is stable for $R_t$.
We denote by $L_0$ is the Ornstein-Uhlenbeck  operator
\[
   L_0\varphi(x)=\frac12\textrm{Tr}\big[D^2\varphi(x)\big]+\langle x,AD\varphi(x)\rangle,\quad \varphi\in \mathcal E_A(H),\, x\in H.
\]
Notice that since $D\varphi(x)\in D(A)$, we have $L_0\varphi\in \mathcal C_{b,1}(H)$.
The next result follows by Proposition 6.2 of \cite{Manca07a}
\begin{Proposition} \label{p.L.4.3}
 For any $\varphi \in \mathcal E_A(H)$ we have   
$\varphi\in D(L,\mathcal C_{b,1}(H))$ and 
\begin{equation} \label{e.L.44}
     L\varphi(x)= L_0\varphi(x),\quad x\in H.  
\end{equation}
The set $\mathcal E_A(H)$ is a $\pi$-core for 
$ (L,D(L,\mathcal C_{b,1}(H)))$, and for any $\varphi\in D(L,\mathcal C_{b,1}(H)) $  there exist $m\in\Nset$ and an $m$-indexed sequence $(\varphi_{n_1,\ldots,n_m})_{n_1,\ldots,n_m\in\Nset}\subset \mathcal E_A(H)$ such that 
\begin{eqnarray} \label{e.L.45} 
 && \lim_{n_1\to \infty}\cdots  \lim_{n_m\to \infty}
     \frac{\varphi_{n_1,\ldots,n_m}}{1+|\cdot|_2}\stackrel{\pi}{=}\frac{\varphi}{1+|\cdot|_2},\\
  \label{e.L.46} 
   && \lim_{n_1\to \infty}\cdots  \lim_{n_m\to \infty} 
\frac{L_0\varphi_{n_1,\ldots,n_m}}{1+|\cdot|_2} \stackrel{\pi}{=}\frac{L\varphi}{1+|\cdot|_2}.
\end{eqnarray} 
Finally, if $\varphi\in D(L,\mathcal C_{b,1}(H))\cap C_b^1(H) $ we can choose the sequence in such a way that \eqref{e.L.45}, \eqref{e.L.46} hold and
\begin{equation} \label{e.L.47} 
\lim_{n_1\to \infty}\cdots  \lim_{n_m\to \infty} 
\langle D \varphi_{n_1,\ldots,n_m},h\rangle \stackrel{\pi}{=}\langle D \varphi ,h\rangle,
\end{equation}
for any $h\in H$.
\end{Proposition}
\begin{proof}
Proposition 6.2 of \cite{Manca07a} shows that \eqref{e.L.44} holds when the OU semigroup and the corresponding generator is considered in the space con {\em uniformly} continuous functions. 
However, one can see that the all the approximations hold also in $\mathcal C_{b,1}(H)$.
\end{proof}

The following results are proved in \cite{Manca07a}  in  spaces of uniformly continuous functions instead of $\mathcal C_{b,1}(H)$. 
However, one can see that the proof  can be easy adapted to $\mathcal C_{b,1}(H)$.

\begin{Proposition} \label{p.L.2.9}
Let $R_t$, $t\geq0$ be the OU semigroup \eqref{e.B.23a} in the space
$\mathcal C_{b,1}(H)$ 
and let $(L,D(L,\mathcal C_{b,1}(H))) $   
be its infinitesimal generators, defined by  \eqref{e.OU.0}.  
Then
\begin{itemize}
\item[(i)] $R_t\in \mathcal L( C_{b,1}(H))$ and  $R_t R_s=  R_{t+s}$, for any $t,s\geq0$;
\item[(ii)] for any $\varphi\in C_{b,1}(H)$ and any sequence $(\varphi_n)_{n\in \Nset}\subset C_{b,1} (H)$ such that 
\[
  \lim_{n\to\infty} \frac{\varphi_n}{1+|\cdot|^{}}  \stackrel{\pi}{=}\frac{\varphi }{1+|\cdot|^{}} 
\]
we have, for any $t\geq0$,  
\[
  \lim_{n\to\infty} \frac{{P}_t\varphi_n}{1+|\cdot|^{}}  \stackrel{\pi}{=}\frac{{P}_t\varphi }{1+|\cdot|^{}}; 
\]
\item[(iii)] for any $\varphi \in D(L,\mathcal C_{b,1}(H))$, we have $  R_t\varphi \in D(L,\mathcal C_{b,1}(H)))$ and 
     $  LR_t\varphi =  R_t  L\varphi$, $t\geq0$;
\item[(iv)] for any $\varphi \in D(L,\mathcal C_{b,1}(H))$, $x\in H$, the map $[0,\infty)\to \Rset$, $t\mapsto  R_t\varphi(x)$ is continuously differentiable and $(d/dt) R_t\varphi(x) =  R_t L\varphi(x)$;
%
\item[(v)]  
 for any $\lambda>\omega_0$ the linear operator $R(\lambda, L)$ on $\mathcal C_{b,1}(H)$ defined by
\[
   R(\lambda, L)f(x)=\int_0^\infty e^{-\lambda t} R_tf(x)dt, \quad f\in \mathcal C_{b,1}(H),\,x\in H
\]
satisfies, for any  $f\in C_{b,1}(H)$
\[
   R(\lambda, L)\in \mathcal L(\mathcal C_{b,1}(H)),
\]
\[
  R(\lambda, L)f\in D(L,\mathcal C_{b,1}(H)),\quad (\lambda I- L)R(\lambda, L)f=f.
\]
We call $R(\lambda, L)$ the {\em resolvent} of $  L$ at $\lambda$;
\item[(vi)] for any $\varphi\in \mathcal C_{b,1}(H)$, $t>0$, the function 
\[
   H\to \Rset,\quad x\mapsto \int_0^tR_s\varphi(x)ds
\]
belongs to $D(L,\mathcal C_{b,1}(H))$, and it holds 
\[
   L\left(  \int_0^tR_s\varphi ds \right)=R_t\varphi-\varphi.
\] 
\end{itemize}
\end{Proposition}

We now study the OU semigroup $R_t$ in the space $C_{b,V}(L^6(0,1))$.
Proposition \ref{p.B.15} implies
that for any $T>0$ there exists $c_T>0$ such that 
\begin{equation}  \label{e.B.22}
  \EE\left[\sup_{t\in [0,T]}V(Z(s,x))\right]\leq c_T\left(1+V(x)\right).
\end{equation}
%
%
%
%
%
Clearly, \eqref{e.B.22} shows that $R_t$ acts on $\mathcal C_{b,V}(L^6(0,1))$.
It is obvious that all the results of Proposition \ref{p.B.2.1} holds also for the OU semigroup  $R_t$, ${t\geq0}$.
We define the infinitesimal generator of $R_t$, ${t\geq0}$ in $\mathcal C_{b,V}(L^6(0,1))$ by setting
\begin{equation}  \label{e.B.23}
\begin{cases}
\DS D(L_V,\mathcal C_{b,V}(L^6(0,1)))=\bigg\{ \varphi \in \mathcal C_{b,V}(L^6(0,1)): \exists g\in \mathcal C_{b,V}(L^6(0,1)),    \\ 
  \DS  \qquad\lim_{t\to 0^+} \frac{ R_t\varphi(x)-\varphi(x)}{t}=g(x),\,\forall x\in L^6(0,1),\;\sup_{t\in(0,1)}\left\|\frac{R_t\varphi-\varphi}{t}\right\|_{0,V}<\infty \bigg\}\\   
   {}   \\
  \DS  L_V\varphi(x)=\lim_{t\to 0^+} \frac{ R_t\varphi(x)-\varphi(x)}{t},\quad \varphi\in D(L_V,\mathcal C_{b,V}(L^6(0,1))),\,x\in L^6(0,1).
\end{cases}
\end{equation}
\begin{Remark}\label{r.B.13}
{\em 
 Since all the results of Proposition \ref{p.B.2.1} hold for the OU semigroup, it follows that all the results of Proposition \ref{p.B.2.9} hold for the OU semigroup and its infinitesimal generator in $\mathcal C_{b,V}(L^6(0,1))$.
 }
\end{Remark}
%
\begin{Proposition} \label{p.B.6.4}
We have  $\mathcal E_A(H)\subset D(L_V,\mathcal C_{b,V}(L^6(0,1)))$, and $L_V\varphi(x)=L\varphi(x)=L_0\varphi$, for any $\varphi \in \mathcal E_A(H)$, $x\in L^6(0,1)$.
\end{Proposition}
\begin{proof}
Since $\mathcal C_{b,1}(H)\subset \mathcal C_{b,V}(L^6(0,1))$ with continuous embedding, it follows
$D(L,\mathcal C_{b,1}(H))\subset D(L_V, \mathcal C_{b,V}(L^6(0,1))) $.
Then by Proposition \ref{p.B.6.4} we have $\mathcal E_A(H) \subset \mathcal C_{b,V}(L^6(0,1))$ and
$L_V\varphi(x)=L_0\varphi(x)$, $\varphi\in\mathcal E_A(H)  $, $x\in L^6(0,1)$.
\end{proof}

We conclude this section by the following formula, which will be often recalled throughout the paper
\begin{equation} \label{e.B.diff}
  \langle DR_t\varphi(x),h\rangle = R_t\left(\langle D\varphi, e^{tA}h\rangle   \right)(x),\quad t\geq0,\,\varphi\in \mathcal C_b^1(H),\,x,h\in H.
\end{equation}
Clearly, by the above formula it follows that $ \mathcal C_b^1(H)$ is stable under $R_t$.
\section{The approximated problem}
 It is convenient to consider the usual Galerkin approximations 
of equation \eqref{e.B.3}. 
For any $m\in\Nset$ we define
\[
   b_m(x)=P_m b(P_mx),\quad x\in H 
\]
where
\[
  P_m=\sum_{i=1}^me_i\otimes e_i,\quad m\in \Nset. 
\]
We consider the approximating problem 
\begin{equation} \label{e5.11}
\left\{\begin{array}{lll}
     dX^m(t)&=&(AX^m(t)+b_m(X^m(t))dt+dW(t),\\\\
     X^m(0)&=&x,
 \end{array}\right. 
\end{equation}
By setting $Y^m(t,x)=X^m(t,x)-W_A(t)$, the corresponding mild form is
\begin{equation} \label{e.B.16}
   Y^m(t,x)=e^{tA}x+\frac{1}{2}\int_{0}^{t}e^{(t-s)A} P_m  D_\xi 
          \left(P_m(Y^m(s,x)+W_A(s))\right)^2  ds,
\end{equation}
Since for any $m\in \Nset$ the identity
\[
   \langle b_m(x),x\rangle =0,\quad x\in H
\]
holds, 
  all the estimates of Proposition \ref{p.B.7}, \ref{p.B.5} are uniform on $m$ and we have the following result.   
\begin{Theorem} \label{t.B.10}   
For any $x \in L^p(0,1),\,p\in [2,\infty) $  there exists a unique mild solution $X^m \in L^p(0,1)$ of
equation \eqref{e5.11}. 
Moreover, for any $x_0\in L^p(0,1)$, $\delta>0$ and  $T>0$ 
\[
   \lim_{m\to\infty}\sup_{\stackrel{|x-x_0|_p<\delta}{t\in [0,T]}}|X^m(t,x)-X(t,x)  |_p =0
\]
\end{Theorem}
We   denote by   $P_t^m$ the transition semigroup  
\begin{equation}   \label{e.B.19}
      P^m_t\varphi(x)=\EE[\varphi(X^m(t,x))],\quad t\geq0,\,\varphi\in \mathcal C_{b,V}(L^6(0,1)), \, x\in L^6(0,1)
\end{equation}
By a standard argument, we find that for any $\mathcal C_{b,V}(L^6(0,1))$ we have 
\[
     \lim_{m\to \infty}\frac{P^m_t\varphi}{1+V}\stackrel{\pi}{=}\frac{P_t\varphi}{1+V},\quad t\ge 0. 
\]
For any $m\in \Nset$, we define the infinitesimal generator of the semigroup $P_t^m$, $t\geq0$ by
\begin{equation}  \label{e.B.20}
\begin{cases}
\DS D(K_m,\mathcal C_{b,V}(L^6(0,1)))=\bigg\{ \varphi \in \mathcal C_{b,V}(L^6(0,1)): \exists g\in \mathcal C_{b,V}(L^6(0,1)),   \\ 
  \DS  \qquad \lim_{t\to 0^+} \frac{ P^m_t\varphi(x)-\varphi(x)}{t}=g(x),\,x\in L^6(0,1),\;\sup_{t\in(0,1)}\left\|\frac{P^m_t\varphi-\varphi}{t}\right\|_{0,V}<\infty \bigg\}\\   
   {}   \\
  \DS  K_m\varphi(x)=\lim_{t\to 0^+} \frac{ {P}_t\varphi(x)-\varphi(x)}{t},\quad
   \varphi\in D(K_m,\mathcal C_{b,V}(L^6(0,1))),\,x\in L^6(0,1).
\end{cases}
\end{equation}
It is clear that all the results of Propositions \ref{p.B.2.1}, \ref{p.B.2.9} hold for  $P_t^m$, ${t\geq0}$ and for its infinitesimal generator $(K_m,D(K_m,\mathcal C_{b,V}(L^6(0,1))))$.
 %
%

%
%
%
%
%
%
%
%
%
%
%
%
%
\subsection{The differential   $DP^m_t\varphi$}
Usually, one derives estimates on the differential
$DP_t\varphi$ directly from the estimates of 
the differential $X_x(t,x)$ of the solution $X(t,x)$.
This method cannot be applied here, by the lack of informations about $X_x(t,x)$. 
In \cite{DPD07}, it is proposed to consider a Kolmogorov operator with an additional  potential term
\[
   K_0'\varphi(x)=K_0\varphi(x)-c|x|_{4}^4\varphi(x),\quad \varphi\in \mathcal E_A(H)
\]
and the corresponding semigroup given by the  Feynman-Kac formula
\[
  S_t\varphi(x)=\EE\left[e^{-c\int_0^t|X(s,x)|_{4}^4ds} \varphi(X(t,x))  \right].
\]
By using a generalization of the Bismut-Elworthy formula (see \cite{DPZ97}) and some estimates on $X_x(t,x)$ the authors are able to get estimates on $DS(t)\varphi$.
Then, by the formula
\[
  P_t\varphi=S_t\varphi+c\int_0^tS_{t-s}\left(|\cdot|_{4}^4\varphi  \right)ds
\]
they get estimates on $DP_t\varphi$.

This method  has been successfully implemented to get solutions  for the $3D$-Navier-Stokes equation (see \cite{DPD03}, \cite{DO}).
It has been also used to get smoothing properties of the differential $DP_t\varphi$, with applications to control problems (see, for instance,    \cite{DPD00}, \cite{DPD00a}, \cite{Mancacontrol07}) 

The following result is proved in Proposition 3.6 of \cite{DPD07}.  
\begin{Proposition} \label{p.B.5}
There exists $\omega_1>0$ such that for any $m\in\Nset$, $t>0$ and $\varphi\in \mathcal C_b^1(H)$ with 
$ D\varphi \in {C_b(H;H^1(0,1))}$ we have $ DP_t^m\varphi(x)\in H^1(0,1)$ and
\[
    |  DP_t^m\varphi(x)|_{H^1(0,1)}\leq
      \left(\|D\varphi\|_{C_b(H;H^1(0,1))}+c\|\varphi\|_0\right)\left(1+|x|_{6}   \right)^8e^{\omega_1t}
\]
\end{Proposition}
The following two results are essential for the proof of Theorem \ref{T.B.2}.
\begin{Proposition} \label{p.B.17a}
 Take $\lambda > \omega_0,\,\omega_1$, where $\omega_0$ is as in Proposition \ref{p.B.2.1} and $\omega_1$ is as in Proposition \ref{p.B.5}. 
Let   $f\in \mathcal E_A(H)$ and,  for $m\in \Nset$ consider the function 
\[
    L^6(0,1)\to \Rset, \quad x\mapsto  \varphi(x)=\int_0^\infty e^{-\lambda t}P_t^m f(x)dt.
\]
Then
\begin{itemize}
\item[(i)] $\varphi$
is continuous, bounded and Fr\'echet differentiable in any $x\in L^6(0,1)$ with continuous differential 
$ D\varphi\in \mathcal C(L^6(0,1);H^1(0,1))$. 
Moreover, it holds
\begin{equation} \label{e.B.28a}
| D \varphi(x)|_{H^1(0,1)}\leq   
 \frac{1}{\lambda-\omega_1}\,\left(\|Df\|_{C_b(H;H^1(0,1))}+c\|f\|_0\right)\left(1+|x|_{6}   \right)^8; 
\end{equation}
\item[(ii)] $\varphi$ belongs to $D(L_V,\mathcal C_{b,V}(L^6(0,1)))\cap D(K_m,\mathcal C_{b,V}(L^6(0,1)))$ and  
\begin{equation}\label{e.B.27a}
     K_m\varphi(x)=L_V\varphi(x)-\frac12\left\langle D_\xi P_m D\varphi(x), (P_mx)^2\right\rangle,\quad\forall x\in L^6(0,1).
\end{equation}
\end{itemize}
\end{Proposition}
\begin{proof}
Notice that the mild solution of \eqref{e5.11} is defined for any $x\in L^2(0,1)=H$ (cfr. \cite{DPDT}).
Then, the transition semigroup $P_t$ can be defined in $\mathcal C_b(H)$. 
So, since $f\in C_b(H)$, it follows $\varphi\in \mathcal C_b(H)$. 
By Proposition \eqref{p.B.5} we have 
\begin{eqnarray*}
  | D\varphi(x)|_{H^1(0,1)} &\leq& \int_0^\infty e^{-\lambda t} |DP_tf(x)|_{H^1(0,1)}   dt 
\\
   &\leq& \int_0^\infty e^{-(\lambda-\omega_1) t}   dt
      \left( \|Df\|_{C_b(H;H^1(0,1))}+c\|f\|_0\right)\left(1+|x|_{6}   \right)^8
\end{eqnarray*}
and \eqref{e.B.28a} follows.
Still by \eqref{e.B.28a} we get $ D\varphi\in \mathcal C(L^6(0,1);H^1(0,1))$.
Indeed, for any $x,h\in L^6(0,1)$,
\[
   |  D\varphi(x+h)-D_\xi D\varphi(x)|_{H^1(0,1)}
\]
\[
   \leq  
   \frac{1}{\lambda-\omega_1}\,\left(\|Df(\cdot+h)-Df(\cdot) \|_{C_b(H;H^1(0,1))}+c\|f(\cdot+h)-f(\cdot)\|_0\right) \left(1+|x|_{6}   \right)^8
\]
Since $f\in C_b(H)$, and  $ Df\in C_b(H;H^1(0,1))$, by uniform continuity it follows $|  D\varphi(x+h)-  D\varphi(x)|_{H^1(0,1)}\to 0$ as $|h|_6\to 0$.
This concludes the proof of (i).

Let us prove (ii).
%
%
%
%
%
%
%
Since the semigroup $P_t^m$, ${t\geq0}$ satisfies the statements of Proposition \ref{p.B.2.1}, it follows that its infinitesimal generator $ K_m$   enjoys the statements of Proposition \ref{p.B.2.9}.
In particular, we have $\varphi=R(\lambda,K_m)f$ and therefore  $\varphi\in D(K_m,\mathcal C_{b,V}(L^6(0,1)))$.
Then we have to show that $\varphi\in D(L_V,\mathcal C_{b,V}(L^6(0,1)))$. 
Now let $R_t$, ${t\geq0}$ be the OU semigroup \eqref{e.B.23a} and let $(L_V,D(L_V,\mathcal C_{b,V}(L^6(0,1))))$ be its infinitesimal generator in $\mathcal C_{b,V}(L^6(0,1))$.
Fix $x\in L^6(0,1)$, $T>0$ and for $t\in [0,T]$ set $X^m(t)=X^m(t,x)$, $Z(t)=Z(t,x)$.
By \eqref{e.B.16}, \eqref{e.B.3.1} we have 
\[
  X^m(t)=Z(t)+\frac{1}{2}\int_{0}^{t}e^{(t-s)A}P_m D_\xi (P_mX^m(s) )^2 ds 
\] 
and consequently 
\[
   P_t^m\varphi(x)=\EE\left[\varphi(X^m(t))  \right]=\EE\left[\varphi(Z(t)+\frac{1}{2}\int_{0}^{t}e^{(t-s)A}P_m D_\xi (P_mX^m(s) )^2 ds)  \right].
\]
Notice that since $f\in C_b^1(H)$, by \eqref{e.B.28a} we get that   the function $L^6(0,1)\to \Rset$, 
$x\mapsto  D\varphi(x)$ is  continuous. 
Then, by Taylor formula we have
\[
  R_t\varphi(x)-\varphi(x)= P_t^m\varphi(x)-\varphi(x)
\]
\begin{equation} \label{e.B.27}
+\frac{1}{2}\EE\left[\int_0^1\Big\langle D\varphi(\xi X^m(t)+(1-\xi)Z(t)),
    \int_{0}^{t} e^{(t-s)A}P_m D_\xi  (P_m X^m(s) )^2 ds\Big\rangle d\xi\right]
\end{equation}
We claim that 
\[
\lim_{t\to 0^+}\frac1t\EE\left[\int_0^1\Big\langle D\varphi(\xi X^m(t)+(1-\xi)Z(t)),\int_{0}^{t} e^{(t-s)A}P_m D_\xi  (P_m X^m(s) )^2 ds\Big\rangle d\xi\right]
\]
\begin{equation} \label{e.B.27d}
   = - \langle D_\xi  P_m D\varphi(x),    (P_mx)^2 \rangle
\end{equation}
holds.
By Theorem \ref{t.B.7.1}, for any $T>0$ we can write 
\[
\begin{split}
   X(t)&=x+\theta_1(t)\\
   Z(t)&=x+\theta_2(t),\quad t\in[0,T]
\end{split}
\]
 where $\theta_1(t),\,\theta_2(t):\Omega\to H,\,t\in[0,T]$ are random variables such that $\theta_1,\theta_2\in C([0,T];H)$ $\PP$-a.s. and $ \theta_1(0)=\theta_2(0)=0$.
On the other hand, by Proposition \ref{p.B.17a} we can write 
\[
   D\varphi(x+z)= D\varphi(x)+\eta(z),\quad z\in H
\]  
where $\eta\in \mathcal C(H,H^1(0,1))$ and $\eta(0)=0$.
With these notations we have
\begin{eqnarray*}
    D\varphi(\xi X^m(t)+(1-\xi)Z(t))&=& D\varphi(x+\xi \theta_1(t)+(1-\xi)\theta_2(t))
\\
   &=&D\varphi(x)+\eta(\xi \theta_1(t)+(1-\xi)\theta_2(t)).
\end{eqnarray*}
Then
\[
 \lim_{t\to 0^+} \sup_{\xi\in [0,1]} |D\varphi(\xi X^m(t)+(1-\xi)Z(t))-D\varphi(x)|_{H^1(0,1)}=
\] 
\begin{equation}  \label{e.B.26a}
  = \lim_{t\to 0^+}\sup_{\xi\in [0,1]} \big|\eta\big(\xi \theta_1(t)+(1-\xi)\theta_2(t)\big)\big|_{H^1(0,1)}=0.
\end{equation}
For any $t>0$
we have
\begin{eqnarray}
    && \left|\frac1t\int_{0}^{t} e^{(t-s)A}P_m D_\xi  (P_m X^m(s) )^2 ds-P_m D_\xi((P_mx)^2)  \right|_{W^{-1,2}(0,1)}
     \notag
\\
 &&\qquad \leq \frac1t\int_{0}^{t}\left| e^{(t-s)A}P_m D_\xi \left( (P_m X^m(s) )^2  -(P_mx)^2\right)\right|_{W^{-1,2}(0,1)}ds
       \notag
\\
 \label{e.B.27b}
  &&\qquad\quad + \frac1t\int_{0}^{t}\left| e^{(t-s)A} P_m D_\xi (P_mx )^2  -P_m D_\xi (P_mx)^2\right|_{W^{-1,2}(0,1)}ds.
\end{eqnarray}
The first term on the right-hand side is bounded by 
\begin{eqnarray*}
 && \frac1t\int_{0}^{t}\left| P_m D_\xi \left( (P_m X^m(s) )^2  -(P_mx)^2\right)\right|_{W^{-1,2}(0,1)}ds
\\
 &&\qquad  \leq\frac1t\int_{0}^{t}\left|   (P_m X^m(s) )^2  -(P_mx)^2 \right|_2ds
\\
 &&\qquad  \leq \frac1t\int_{0}^{t}\left|   X^m(s)    - x  \right|_2\left|   X^m(s)    + x  \right|_2ds.
\end{eqnarray*}
Since $X^m\in C([0,T];H)$ $\PP$-a.s., it follows
\[
  \lim_{t\to0^+} \frac1t\int_{0}^{t}\left| e^{(t-s)A}P_m D_\xi \left( (P_m X^m(s) )^2  -(P_mx)^2\right)\right|_{W^{-1,2}(0,1)}ds=0,\quad \PP\text{-a.s.}
\]
Since the semigroup $e^{tA},\,t\geq0$ can be est ended to a strongly continuous semigroup in  $W^{-1,2}(0,1)$, for the last term of \eqref{e.B.27b} it holds
\[
 \lim_{t\to0^+}  \frac1t\int_{0}^{t}\left| e^{(t-s)A} P_m D_\xi (P_mx )^2 
  -P_m D_\xi (P_mx)^2\right|_{W^{-1,2}(0,1)}ds=0.
\]
Hence, by \eqref{e.B.27b} we have
\[
  \lim_{t\to0^+} \left|\frac1t\int_{0}^{t} e^{(t-s)A}P_m D_\xi  (P_m X^m(s) )^2 ds-P_m D_\xi(P_mx)^2 \right|_{W^{-1,2}(0,1)}=0,\quad \PP\text{-a.s.}
\]
This, together with \eqref{e.B.27a} and an integration by parts, implies
\[
 \lim_{t\to0^+}  \frac1t \int_0^1\Big\langle D\varphi(\xi X^m(t)+(1-\xi)Z(t)),\int_{0}^{t} e^{(t-s)A}P_m D_\xi  (P_m X^m(s) )^2 ds\Big\rangle d\xi 
\]
\begin{equation} \label{e.B.30a}
   = \langle D\varphi(x), P_m D_\xi  (P_mx)^2 \rangle= -\langle D_\xi P_m D\varphi(x),   (P_mx)^2 \rangle,\quad \PP\text{-a.s.}
\end{equation}
In order to obtain \eqref{e.B.26a}, it is sufficient to show that the terms in the above limit are dominated by an integrable random variable.
Indeed, for any $t\in (0,T]$ we have
\begin{eqnarray*}
 &&\frac1t \int_0^1\Big\langle D\varphi(\xi X^m(t)+(1-\xi)Z(t)),\int_{0}^{t} e^{(t-s)A}P_m D_\xi  (P_m X^m(s) )^2 ds\Big\rangle d\xi 
\\
  &&\qquad\qquad\leq \frac1t\left|\int_0^1D\varphi(\xi X^m(t)+(1-\xi)Z(t))d\xi\right|_{H^1(0,1)}
\\
 &&\hspace{100pt}\qquad\times \left|\int_{0}^{t} e^{(t-s)A}P_m D_\xi  (P_m X^m(s) )^2 ds\right|_{W^{-1,2}(0,1)}
\\
  &&\qquad\qquad\leq\int_0^1\left|D\varphi(\xi X^m(t)+(1-\xi)Z(t))\right|_{H^1(0,1)}d\xi 
\\
 &&\hspace{100pt}\qquad \times
    \frac1t\int_{0}^{t} \left|e^{(t-s)A}P_m D_\xi  (P_m X^m(s) )^2 \right|_{W^{-1,2}(0,1)}ds
\\
   &&\qquad\qquad\leq I_1(t)\times I_2(t).
\end{eqnarray*}
Set
\[
   C_\varphi=\left(\frac{ \|Df\|_{C_b(H;H^1(0,1))}+c\|f\|_0}{\lambda-\omega_1}\right).
\]
By \eqref{e.B.28a} we have  
\begin{eqnarray*}
  && \int_0^1| D\varphi(\xi X^m(t,x)+(1-\xi)Z(t,x))|_{H^1(0,1)}d\xi
\\
 && \qquad  \leq  C_\varphi \,
      \left(1+    |\xi X^m(t,x)+(1-\xi)Z(t,x) |_{6}^{8}   \right)
\\
 && \qquad\leq  C_\varphi\,
    \int_0^1 \left(1+ \xi   |X^m(t,x)|_{6}^{8}  +(1-\xi)   |Z(t,x)|_{6}^{8}  \right) d\xi
\\
 && \qquad \leq  C_\varphi     
      \left(1+ \sup_{t\in[0,T]}   |X^m(t,x)|_{6}^{8}  +\sup_{t\in[0,T]}    |Z(t,x)|_{6}^{8}  \right).
\end{eqnarray*}
Here we have used the convexity of the function $z\to |z|_{6}^{8}$. 
For $I_2(t)$ we have
\begin{eqnarray*}
 I_2(t)&\leq& \frac ct\int_{0}^{t} \left|  (P_m X^m(s) )^2 \right|_2ds
\\
   &\leq& \frac ct\int_{0}^{t} \left|  X^m(s)  \right|_4^2 ds \leq c\sup_{t\in [0,T]}\left|  X^m(t)  \right|_4^2
\end{eqnarray*}
\bigskip
Then, for any $t\in (0,T]$ we have
\[
\frac1t\left|\int_0^1\Big\langle D\varphi(\xi X^m(t)+(1-\xi)Z(t)),\int_{0}^{t} e^{(t-s)A}P_m D_\xi  (P_m X^m(s) )^2 ds\Big\rangle d\xi\right|
\]
\begin{equation} \label{e.B.31a}
  \leq        
     c C_\varphi \left(1+ \sup_{t\in[0,T]}   |X^m(t,x)|_{6}^{8}  +\sup_{t\in[0,T]}    |Z(t,x)|_{6}^{8}  \right)\left(  \sup_{t\in [0,T]}\left|  X^m(t)  \right|_4^2 \right) 
\end{equation}
Notice  that by Propositions \ref{p.B.7},   \eqref{p.B.15}   the random variable
\begin{equation}\label{e.B.32a}
   g(x):= c C_\varphi \left(1+ \sup_{t\in[0,T]}   |X^m(t,x)|_{6}^{8}  +\sup_{t\in[0,T]}    |Z(t,x)|_{6}^{8}  \right)\left(  \sup_{t\in [0,T]}\left|  X^m(t)  \right|_4^2 \right)
\end{equation}
belongs to $L^1(\Omega,\PP)$ and 
\begin{equation}\label{e.B.33a}
  \EE[g(x)]\leq C\left(1+  |x|_{6}^{8}|x|_{4}^{2}\right)
\end{equation}
for some $C>0$.
Consequently, since for any $t\in (0,T]$
\[
\frac1t\left|\int_0^1\Big\langle D\varphi(\xi X^m(t)+(1-\xi)Z(t)),\int_{0}^{t} e^{(t-s)A}P_m D_\xi  (P_m X^m(s) )^2 ds\Big\rangle d\xi\right| \leq g(x),
\]
  by the dominated convergence theorem and by \eqref{e.B.30a} follows \eqref{e.B.27d} as claimed. 

By \eqref{e.B.27},   \eqref{e.B.27d} and by the fact that $\varphi\in D(K_m,\mathcal C_{b,V}(L^6(0,1)))$ we obtain
\[
   \lim_{t\to 0^+} \frac{R_t\varphi(x)-\varphi(x)}{t}=K_m\varphi+\frac12\left\langle D_\xi P_m D\varphi(x), (P_mx)^2\right\rangle,\quad \forall x\in L^6(0,1).
\] 
Now, by \eqref{e.B.31a}, \eqref{e.B.32a}, \eqref{e.B.33a} we have
 \begin{eqnarray*}
  \sup_{t\in(0,T]} \left|\frac{R_t\varphi(x)-\varphi(x)}{t}\right|   %
   &\leq&  \sup_{t\in(0,T]}\left|\frac{P_t^m\varphi(x)-\varphi(x)}{t}\right|+\EE[g(x)]
\\
   &\leq& c(1+V(x))
\end{eqnarray*}
since $\varphi\in D(K_m,\mathcal C_{b,V}(L^6(0,1))$. 
This  implies $\varphi\in D(L_V,\mathcal C_{b,V}(L^6(0,1)) ) $ and \eqref{e.B.27a} follows.
\end{proof}
\begin{Proposition} \label{p.B.17aa}
Fix $m\in \Nset$, $f\in \mathcal E_A(H)$ and let $\varphi$ be as in Proposition \ref{p.B.17a}. 
Then, there exist $k\in\Nset$ and   a  $k$-indexed sequence $(\varphi_{n_1,\ldots,n_k})_{n_1\in\Nset,\ldots,n_k\in\Nset}\subset \mathcal E_A(H)$ such that
\begin{equation} \label{e.B.30}
   \lim_{n_1\to \infty}\cdots\lim_{n_k\to \infty}  \frac{ \varphi_{n_1,\ldots,n_k}}{1+V}
       \stackrel{\pi}{=}\frac{  \varphi}{1+V}
\end{equation}
 \begin{equation}\label{e.B.31}
    \lim_{n_1\to \infty}\cdots\lim_{n_k\to \infty}  \frac{L_0\varphi_{n_1,\ldots,n_k}}{1+V}
       \stackrel{\pi}{=}\frac{  L\varphi}{1+V}
 \end{equation}
and, for any  $h\in H$
 \begin{equation}\label{e.B.32}
   \lim_{n_1\to \infty}\cdots\lim_{n_k\to \infty}  
       \frac{\langle  D_\xi D \varphi_{n_1,\ldots,n_k},h\rangle}{\left(1+   |\cdot|_{6}^{8}\right)}  
     \stackrel{\pi}{=} 
     \frac{\langle  D_\xi D \varphi ,h\rangle}{\left(1+   |\cdot|_{6}^{8}\right)} .
 \end{equation}
\end{Proposition}
\begin{proof} 
Set
\[
  \psi_{p}(x)= \left(1+p^{-1}|e^{\frac1pA}x|_{6}^8\right)^{-1} \varphi(e^{\frac1pA}x) ,\quad x\in H,\,p\in\Nset.
\]
Clearly, 
\begin{equation} \label{e.B.33}
   \lim_{p\to\infty} \frac{   \psi_{p} }{1+|\cdot|_{6}^8} 
    \stackrel{\pi}{=}\frac{    \varphi  }{1+|\cdot|_{6}^8}.
\end{equation}
By the well known properties of the heat semigroup, $e^{\frac1pA}x\in L^6(0,1) $, for any $x\in H$.
Then, since by Proposition \ref{p.B.17a} we have $\varphi\in \mathcal C_b(L^6(0,1))$, it follows that $\psi_p:H\to \Rset$ is bounded.
Moreover, an easy computation show that $\psi_p$ is   continuous.
Then, $\psi_p\in\mathcal C_b(H)$.
A standard computation show
\begin{equation} \label{e.B.36a}
   \langle D  \psi_{p}(x),  h\rangle=
     \frac{\langle D\varphi(e^{\frac1pA}x),   e^{\frac1pA}h\rangle}{1+p^{-1}|e^{\frac1pA}x|_{6}^8}
      - \frac{ 8   \varphi(e^{\frac1pA}x) |e^{\frac1pA}x|_{6}^7\langle (e^{\frac1pA}x)^5,  e^{\frac1pA}h\rangle}
            {p \left(1+p^{-1}|e^{\frac1pA}x|_{6}^8\right)^2},
\end{equation}
$x,h\in H$.
Hence, by taking into account \eqref{e.B.28a}, there exists $c_f>0$, depending on $f$, such that for any $x\in L^6(0,1)$ we have  
\begin{eqnarray*}
   |\langle D  \psi_{p}(x),  h\rangle|  &\leq& 
  \frac{| D\varphi(e^{\frac1pA}x)|_2 |e^{\frac1pA}h|_2}{1+p^{-1}|e^{\frac1pA}x|_{6}^8}+     \frac{8\|\varphi\|_0|e^{\frac1pA}x|_{6}^7|e^{\frac1pA}x|_{6}^3|e^{\frac1pA}h|_{L^{6/5}}}
  {p \left(1+p^{-1}|e^{\frac1pA}x|_{6}^8\right)^2}  
\\
 &\leq&  \left( \frac{c_f(1+|e^{\frac1pA}x|_{6}^8)}
                     {\left(1+p^{-1}|e^{\frac1pA}x|_{6}^8\right)(\lambda-\omega_1)} 
    +     \frac{2 \|\varphi\|_0|e^{\frac1pA}x|_{6}^{10}}{ p\left(1+p^{-1}|e^{\frac1pA}x|_{6}^8\right)^2} \right)|e^{\frac1pA}h|_2\\
 &\leq& \left(p \frac{c_f}{\lambda-\omega_1}+2 \|\varphi\|_0\right)|h|_2.
\end{eqnarray*}
Then, $ \psi_p$ is Fr\'echet differentiable in any $x\in H$ and its differential is bounded. 
An easy but tedious computation shows that $D\psi_p:H\to \mathcal L(H)$ is  continuous.  
Therefore, $\psi_p\in \mathcal C_b^1(H)$.
In addition, as easily checked, by formula \eqref{e.B.36a} and by \eqref{e.B.28a} (which gives the bound for $D\varphi$) it follows 
\begin{equation} \label{e.B.34}
   \lim_{p\to\infty} \frac{\langle D  \psi_{p} ,  h\rangle}{1+|\cdot|_{6}^8} 
    \stackrel{\pi}{=}\frac{ \langle D  \varphi,  h\rangle}{1+|\cdot|_{6}^8},\quad \forall h\in H.
\end{equation}
For any $n_2,n_3\in \Nset$, consider the function
\[
  \psi_{n_2,n_3}: H\to \Rset,\quad x\mapsto \psi_{n_2,n_3}(x)={n_2}\int_0^{\frac{1}{n_2}}R_t \psi_{n_3}(x) dt.
\]
By (vi) of Proposition \ref{p.L.2.9} we have  $\psi_{n_2,n_3}\in D(L,\mathcal C_{b,1}(H))$ and by \eqref{e.B.diff} we have $\psi_{n_2,n_3}\in C^1_b(H)$. 
Then
\[
    \psi_{n_2,n_3}\in D(L,\mathcal C_{b,1}(H))\cap\mathcal  C_b^1(H),\quad n_2,n_3\in \Nset. 
\]
Then, by Proposition \ref{p.L.4.3} there exists a sequence\footnote{we assume that the sequence has one index} $\{\psi_{n_2,n_3,n_4}\}_{n_4\in\Nset}\subset \mathcal E_A(H)$ such that
\begin{equation} \label{e.B.35} 
  \lim_{n_4\to\infty}\psi_{n_2,n_3,n_4}\stackrel{\pi}{=}\psi_{n_2,n_3},\quad
   \lim_{n_4\to\infty}L_0\psi_{n_2,n_3,n_4}\stackrel{\pi}{=}L\psi_{n_2,n_3}
\end{equation}
\begin{equation}  \label{e.B.36}
 \lim_{n_4\to\infty}\langle D\psi_{n_2,n_3,n_4},h\rangle \stackrel{\pi}{=}
                  \langle D\psi_{n_2,n_3},h\rangle, \quad \forall h\in H.
\end{equation}
Now set
\begin{equation*}
\begin{split}
   & \varphi_{n_1} = R_{\frac{1}{n_1}}\varphi,
\\
   & \varphi_{n_1,n_2} = {n_2}\int_0^{\frac{1}{n_2}}R_{t+\frac{1}{n_1}} \varphi dt,%
\\
  & \varphi_{n_1,n_2,n_3} = {n_2}\int_0^{\frac{1}{n_2}}R_{t+\frac{1}{n_1}} \psi_{n_3} dt=R_{\frac{1}{n_1}} \psi_{n_2,n_3},
\\
   &\varphi_{n_1,n_2,n_3,n_4}= R_{\frac{1}{n_1}} \psi_{n_2,n_3,n_4}
\end{split}
\end{equation*}
As easily checked, by the definition of $\varphi_{n_1,n_2,n_3,n_4}$ and by \eqref{e.B.33}, \eqref{e.B.33}
\[   
    \lim_{n_1\to\infty}\lim_{n_2\to\infty}\lim_{n_3\to\infty}\lim_{n_4\to\infty}
      \frac{ \varphi_{n_1,n_2,n_3,n_4}}{1+|\cdot|_{6}^8}
     \stackrel{\pi}{=} \frac{\varphi}{1+|\cdot|_{6}^8}.
\]
which implies \eqref{e.B.30}.

Let us show \eqref{e.B.31}.
Since $\psi_{n_2,n_3,n_4}$, by \eqref{e.OU.35} we have that $\varphi_{n_1,n_2,n_3,n_4}\in \mathcal E_A(H)$ and by Proposition \ref{p.L.4.3} we have
\[ 
  L\varphi_{n_1,n_2,n_3,n_4}=L_0\varphi_{n_1,n_2,n_3,n_4},\quad \forall n_1,n_2,n_3,n_4\in \Nset.
\]
Consequently, by \eqref{e.B.35} and by (iii) of Proposition \ref{p.L.2.9}
\[
   \lim_{n_4 \to\infty}L_0 \varphi_{n_1,n_2,n_3,n_4}= \lim_{n_4 \to\infty}R_{\frac{1}{n_1}}L\psi_{n_2,n_3,n_4}
\]
\[
   \stackrel{\pi}{=}R_{\frac{1}{n_1}}L\psi_{n_2,n_3}=LR_{\frac{1}{n_1}}\psi_{n_2,n_3}=L\varphi_{n_1,n_2,n_3}.
\]
By (vi) Proposition \ref{p.L.2.9} we have 
\[
   LR_{\frac{1}{n_1}}\psi_{n_2,n_3}
   =n_2\left(R_{\frac{1}{n_1}+\frac{1}{n_2}}\psi_{n_3}-R_{\frac{1}{n_1}}\psi_{n_3}\right).
\]
Therefore
\begin{eqnarray*}
    \lim_{n_2\to\infty}\lim_{n_3 \to\infty} 
     \frac{L\varphi_{n_1,n_2,n_3}\left(R_{\frac{1}{n_1}+\frac{1}{n_2}}   \psi_{n_3}-R_{\frac{1}{n_1}}\psi_{n_3}\right)}{1+V}
    &\stackrel{\pi}{=}&
    \lim_{n_2\to\infty}  
    \frac{\left(R_{\frac{1}{n_1}+\frac{1}{n_2}}\varphi  -  R_{\frac{1}{n_1}}\varphi  \right)}{1+V}  \\   
     &\stackrel{\pi}{=}&     \frac{ R_{\frac{1}{n_1}}L\varphi}{1+V}.
\end{eqnarray*}
The last equality follows by (v) of Proposition \ref{p.B.2.1} and by the fact that $\varphi\in D(L_V,\mathcal C_{b,V}(L^6(0,1)))$.
Finally, since
\[
  \lim_{n_1\to\infty}  \frac{ R_{\frac{1}{n_1}}L\varphi}{1+V}\stackrel{\pi}{=} \frac{L\varphi}{1+V},
\]
\eqref{e.B.31} follows.

Let us show \eqref{e.B.32}.
Notice that for any $n_1,n_2,n_3,n_4\in \Nset $, $h\in H^1_0(0,1)$ we have
\begin{eqnarray*}
    \langle D\varphi_{n_1,n_2,n_3,n_4}(x),D_\xi h\rangle&=& 
      R_{\frac{1}{n_1}}\left(\left\langle e^{\frac{1}{n_1}A}D\psi_{n_2,n_3,n_4} ,D_\xi h\right\rangle\right)(x)
\\
   &=&- R_{\frac{1}{n_1}}\left(\left\langle D_\xi e^{\frac{1}{n_1}A}D\psi_{n_2,n_3,n_4} ,  h\right\rangle\right)(x).
\end{eqnarray*}
Here we have also used formula \eqref{e.B.diff}.
By the elementary properties of the heat semigroup, for any $t>0$ the linear operator $D_\xi e^{tA}:H^1_0(0,1)\to H$, $z\mapsto D_\xi e^{tA}  z$ is bounded by $  |D_\xi e^{tA} z|_2\leq c t^{-1/2}|z|_2$, 
where   $c>0$ is independent of $t$. 
Then  $D_\xi e^{\frac{1}{n_1}A}:H^1_0(0,1)\to H$ can be extended to a linear and bounded operator in $H$, which we still denote  by $D_\xi e^{\frac{1}{n_1}A}$.
Then by \eqref{e.B.36} we have %
\begin{equation*}  %
 \lim_{n_4\to\infty}\langle D_\xi   D\varphi_{n_1,n_2,n_3,n_4},h\rangle \stackrel{\pi}{=}
                  \langle D_\xi   D\varphi_{n_1,n_2,n_3},h\rangle, \quad \forall h\in H.
\end{equation*}
By the same argument we find
\[
     \lim_{n_3\to\infty}\langle D_\xi   D\varphi_{n_1,n_2,n_3},h\rangle \stackrel{\pi}{=}
                  \langle D_\xi   D\varphi_{n_1,n_2},h\rangle, \quad \forall h\in H.
\]
Notice now that by definition of $\varphi_{n_1,n_2}$ we have
\[
    \langle D_\xi D\varphi_{n_1,n_2}(x), h\rangle= 
      R_{\frac{1}{n_1}}\left(\left\langle D_\xi e^{\frac{1}{n_1}A}D\psi_{n_2} , h\right\rangle\right)(x),\quad x,h\in H.
\]
Now, since   $D_\xi e^{\frac{1}{n_1}A}:H\to H $ is linear and bounded, by \eqref{e.B.34} it follows
\[
    \lim_{n_2\to\infty} 
      \frac{ \big\langle D_\xi e^{\frac{1}{n_1}A}D\psi_{n_2}, h\big\rangle}{1+|\cdot|_{6}^8}
        \stackrel{\pi}{=}\frac{ \big\langle D_\xi e^{\frac{1}{n_1}A}D\varphi, h\big\rangle}
          {1+|\cdot|_{6}^8}.
\]
Hence, by Proposition \ref{p.B.2.1} (cfr. Remark \ref{r.B.13}) we have  
\[
    \lim_{n_2\to\infty}\frac{ \langle D_\xi D\varphi_{n_1,n_2}, h\rangle}{1+|\cdot|_{6}^8}
                        \stackrel{\pi}{=}\frac{\langle D_\xi D\varphi_{n_1} , h\rangle}{1+|\cdot|_{6}^8}.
\]
Finally, by Proposition \ref{p.B.2.1} applied to the semigroup $R_t$, ${t\geq0}$ we find 
\begin{eqnarray*}
    \lim_{n_1\to\infty}  \frac{\langle D_\xi D\varphi_{n_1}, h\rangle}{1+|\cdot|_{6}^8}         
   &=&\lim_{n_1\to\infty} 
         \frac{ R_{\frac{1}{n_1}}\left( \langle D_\xi e^{\frac{1}{n_1}A}D\varphi, h\rangle\right) }{1+|\cdot|_{6}^8}    \\                     
           &\stackrel{\pi}{=}&\frac{ \langle D_\xi D\varphi , h\rangle }{1+|\cdot|_{6}^8}.
\end{eqnarray*}
This complete the proof.
\end{proof}

\section{Main result: a core for $K$} \label{s.B.7}
In order to prove  Theorem \ref{T.B.2} below we show that $K$ is an extension of $K_0$ and that $(\lambda I-K_0)(\mathcal E_A(H))$ is dense, with respect to the $\pi$-convergence, in $\mathcal C_{b,V}(L^6(0,1))$.
The proof follows a scheme which has been successfully implemented in \cite{Manca07} and \cite{Manca07}. 
Here we give all the details.
We stress that a crucial role is played by the estimates for the differential $DP_t$ given in the previous section. 
\begin{Theorem} \label{T.B.2}
The operator $(K,D(K,\mathcal C_{b,V}(L^6(0,1))))$ is an extension of $K_0$, and for any $\varphi\in\mathcal E_A(H)$ we have  $\varphi\in D(K,\mathcal C_{b,V}(L^6(0,1)))$ and $K\varphi=K_0\varphi$.
Finally, the set $ \mathcal E_A(H)$ is a $\pi$-core for $(K,D(K,\mathcal C_{b,V}(L^6(0,1))))$, that is for any $\varphi\in D(K,\mathcal C_{b,V}(L^6(0,1)))$ there exist $m\in\Nset$ and   an $m$-indexed sequence 
 $(\varphi_{n_1,\ldots,n_m})_{n_1\in\Nset,\ldots,n_m\in\Nset}\subset \mathcal E_A(H)$ such that
\begin{equation*} 
   \lim_{n_1\to \infty}\cdots\lim_{n_m\to \infty}  \frac{ \varphi_{n_1,\ldots,n_m}}{1+V} \stackrel{\pi}{=}\frac{ \varphi}{1+V} 
\end{equation*}
and
 \begin{equation*}
\lim_{n_1\to \infty}\cdots\lim_{n_m\to \infty}  \frac{ K_0\varphi_{n_1,\ldots,n_m}}{1+V} \stackrel{\pi}{=}\frac{ K\varphi}{1+V}  
 \end{equation*}
\end{Theorem}

We split the proof into two lemmata.
\begin{Lemma} \label{l.B.18}
  $K$ is an extension of $K_0$ and  $K\varphi=K_0\varphi$ for any $\varphi \in \mathcal E_A(H)$.
\end{Lemma}
\begin{proof}
Take $h\in D(A)$.
It is sufficient to show the claim for 
\[
   \varphi(x)= e^{i\langle x,h \rangle},\quad x\in L^6(0,1).
\]
Let $(L,D(L,\mathcal C_{b,1}(H)))$ be the infinitesimal generator in $\mathcal C_{b,1}(H)$ of the Ornstein-Uhlenbeck semigroup associated to the mild solution of \eqref{e.B.3.1} 
and, for any $m\in \Nset$, let $(K_m,D(K_m,\mathcal C_{b,V}(L^6(0,1))))$ be the infinitesimal generator of the semigroup $P_t^m$, ${t\geq0}$ in $\mathcal C_{b,V}(L^6(0,1))$, as defined in \eqref{e.B.19}, \eqref{e.B.20}.
Since $\mathcal E_A(H)\subset D(L,\mathcal C_b^1(H))\cap\mathcal  C_{b,1}(H)$, 
by arguing as for Proposition \ref{p.B.17a}  we find that for any $t\geq0$, $x\in L^6(0,1)$ it holds
\[
   P_t^m\varphi(x) - R_t\varphi(x) 
\]
\[
 =  \frac{i}{2}\EE\left[\int_0^1  \varphi(\xi Z (t,x)+(1-\xi)X^m(t,x))d\xi \left\langle h, 
   \int_0^t e^{(t-s)A}P_m D_\xi(P_mX^m(s,x))^2ds\right\rangle  \right]
\]
\[
   =\frac{i}{2}\EE\left[\int_0^1  \varphi(\xi Z (t,x)+(1-\xi)X^m(t,x))d\xi
   \int_0^t \left\langle h,  e^{(t-s)A}P_m D_\xi(P_mX^m(s,x))^2\right\rangle ds d\xi \right]
\]
\[
   =-\frac{i}{2}\EE\left[\int_0^1  \varphi(\xi Z (t,x)+(1-\xi)X^m(t,x))d\xi
   \int_0^t \left\langle D_\xi P_m e^{(t-s)A}h,   (P_mX^m(s,x))^2\right\rangle ds   \right].
\]
In the above computation we have used $D\varphi(x)=i\varphi(x)h$, $x\in L^6(0,1)$.
Letting $m\to\infty$, by Theorem \ref{t.B.10} we find
\[
   P_t\varphi(x)-\varphi(x)=R_t\varphi(x)-\varphi(x)
\]
\[
   -\frac{i}{2}\EE\left[\int_0^1  \varphi(\xi Z (t,x)+(1-\xi)X(t,x))d\xi
   \int_0^t \left\langle D_\xi e^{(t-s)A}h,   (X(s,x))^2\right\rangle ds   \right].
\]
This implies, still by arguing as for Proposition \ref{p.B.17a},
\[
   \lim_{t\to 0^+}\frac{P_t\varphi(x)-\varphi(x)}{t}= L\varphi(x)-\frac{i}{2}\varphi(x)\langle D_\xi h,x^2\rangle=
      L\varphi(x)-\frac{1}{2}\langle D_\xi D\varphi(x),x^2\rangle,
\]
for any $x\in L^6(0,1)$.
As easily seen, $|D_\xi e^{tA}h|_2\leq \pi|D_\xi h|_2$, then
\[
\left|\frac{P_t\varphi(x)-\varphi(x)}{t}\right|\leq
   \left|\frac{R_t\varphi(x)-\varphi(x)}{t}\right|+\frac{|D_\xi h|_2}{2t}\int_0^t\EE[|X(s,x)|_{4}^2]ds
\]
Now, since $\varphi\in D(L,\mathcal C_{b,1}(H))$, the first term of right-hand side is bounded by 
\[
   \left|\frac{R_t\varphi(x)-\varphi(x)}{t}\right|\leq c(1+|x|_2),
\]
where $c_{\varphi,T}>0$ depends only by $\varphi$ and $T$.
By Proposition \ref{p.B.7}, the last term on the right-hand side is bounded by
\[
   \frac{|D_\xi h|_2}{2t}\int_0^t\EE[|X(s,x)|_{4}^2]ds \leq \frac{|D_\xi h|_2}{2 } \EE[\sup_{t\in[0,T]}|X(t,x)|_{4}^2]ds \leq  \frac{c_T|D_\xi h|_2}{2 }(1+|x|_4^2),
\]
where $c_T>0$ depends only by $T$.
Then,
\[
   \sup_{t\in(0,1)}\left\|\frac{P_t\varphi-\varphi}{t}\right\|_{0,V}<\infty.
\]
This implies $\varphi \in D(K,\mathcal C_{b,V}(L^6(0,1)))$ and $K\varphi=L_V\varphi-\frac12\langle D_\xi D\varphi,(\cdot)^2\rangle$.
Consequently, the claim follows by Proposition \ref{p.B.6.4}.
\end{proof}
 \begin{Lemma} \label{l.B.19}
 $\mathcal E_A(H)$ is a $\pi$-core for $(K,D(K,\mathcal C_{b,V}(L^6(0,1))))$, that is for any $\varphi\in D(K,\mathcal C_{b,V}(L^6(0,1)))$ there exist $m\in\Nset$ and   an $m$-indexed sequence $(\varphi_{n_1,\ldots,n_m})_{n_1\in\Nset,\ldots,n_m\in\Nset}\subset \mathcal E_A(H)$ such that
\begin{equation} \label{e.B.37}
   \lim_{n_1\to \infty}\cdots\lim_{n_m\to \infty}  \frac{\varphi_{n_1,\ldots,n_m}}{1+V} \stackrel{\pi}{=}%
   \frac{\varphi}{1+V}
\end{equation}
and
 \begin{equation}\label{e.B.38}
   \lim_{n_1\to \infty}\cdots\lim_{n_m\to \infty} \frac{K_0\varphi_{n_1,\ldots,n_m}}{1+V} \stackrel{\pi}{=}%
   \frac{K\varphi}{1+V}.
 \end{equation}
\end{Lemma}
\noindent
{\bf Step 1}. Take $\varphi\in D(K,\mathcal C_{b,V}(L^6(0,1)))$ and fix $\lambda > \omega_0,\,\omega_1$, where $\omega_0$ is as in Proposition \ref{p.B.2.1} and $\omega_1$ is as in Proposition \ref{p.B.5}. 
We set
$\lambda\varphi - K\varphi=f$.
By Proposition \ref{p.B.2.9} we have $\varphi=R(\lambda,K)f$.
Let us fix a sequence $(f_{n_1})_{n_1\in \Nset}\subset \mathcal E_A(H)$ such that
\[
     \lim_{n_1\to \infty} \frac{f_{n_1}}{1+V} \stackrel{\pi}{=}%
   \frac{f}{1+V}.
\]
We set $\varphi_{n_1}= R(\lambda,K)f_{n_1}$.
By Proposition \ref{p.B.2.1}, \ref{p.B.2.9} it follows
\begin{equation} \label{e.B.39}
   \lim_{n_1\to \infty}   \frac{\varphi_{n_1}}{1+V} \stackrel{\pi}{=}%
   \frac{\varphi}{1+V} ,\quad
   \lim_{n_1\to \infty}   \frac{K\varphi_{n_1}}{1+V} \stackrel{\pi}{=}%
   \frac{K\varphi}{1+V}. 
\end{equation} %
{\bf Step 2}.
Now let $(K_m,D(K_m,\mathcal C_{b,V}(L^6(0,1))))$ be the infinitesimal generator of the semigroup $P_t^m$, ${t\geq0}$ in the space $\mathcal C_{b,V}(L^6(0,1))$, as defined in \eqref{e.B.20}.
We set
\[
  \varphi_{n_1,n_2}=\int_0^\infty e^{-\lambda t} P_t^{n_2}f_{n_1}dt.
\]
By Proposition \ref{p.B.2.9} we have $\varphi_{n_1,n_2}\in D(K_{n_2},\mathcal C_{b,V}(L^6(0,1)))$ and by a standard computation 
\begin{equation} \label{e.B.40}
\lim_{n_2\to \infty}    \frac{ \varphi_{n_1,n_2}}{1+V}\stackrel{\pi}{=} \frac{ \varphi_{n_1}}{1+V}, \quad
\lim_{n_2\to \infty}    \frac{ K_{n_2}\varphi_{n_1,n_2}}{1+V}\stackrel{\pi}{=} \frac{ K\varphi_{n_1}}{1+V}.
\end{equation}
Notice that $f_{n_1}$ satisfies the hypothesis of Proposition \ref{p.B.17a}.
Hence, $\varphi_{n_1,n_2} \in D(L_V,\mathcal C_{b,V}(L^6(0,1)))$ and %
\begin{equation} \label{e.B.41}
   K_{n_2}\varphi_{n_1,n_2}=L\varphi_{n_1,n_2}-\frac12
\left\langle D_\xi P_{n_2} D\varphi_{n_1,n_2}, (P_{n_2}\cdot)^2 \right\rangle,
\end{equation}
for any $n_1,n_2\in \Nset$, $x\in L^6(0,1)$. 
In addition, by \eqref{e.B.28a} it holds  
\begin{eqnarray*}
 &&  \left|\left\langle D_\xi   D\varphi_{n_1,n_2}(x), x^2 \right\rangle-\left\langle D_\xi P_{n_2} D\varphi_{n_1,n_2}(x), (P_{n_2}x)^2 \right\rangle\right|  
\\
&&\qquad=  \left|\left\langle   D\varphi_{n_1,n_2}(x), D_\xi\left(x^2\right)- P_{n_2}D_\xi(P_{n_2}x)^2\right\rangle\right|
\\
&&\qquad\leq \left|   D\varphi_{n_1,n_2}(x)\right|_{H^1(0,1)} 
           \left| D_\xi\left(x^2\right)- P_{n_2}D_\xi(P_{n_2}x)^2  \right|_{W^{-1,2}(0,1)}  \\
&&\qquad\leq  \left(\frac{\|D\varphi_{n_1}\|_{C_b(H;H^1(0,1))}+c\|\varphi_{n_1}\|_0} {\lambda-\omega_1}\right)\left(1+|x|_{6}   \right)^8 \\
&&\qquad\qquad \qquad\times\left| D_\xi\left(x^2\right)- P_{n_2}D_\xi(P_{n_2}x)^2  \right|_{W^{-1,2}(0,1)}
\end{eqnarray*}
for any $x\in L^6(0,1)$, where $W^{-1,2}(0,1)$ is the topological dual of $H^1(0,1)$ endowed with the norm $|\cdot|_{W^{-1,2}(0,1)}$.
Consequently,
\begin{equation} \label{e.B.42}
\lim_{n_2\to\infty }\frac{ \left\langle D_\xi   D\varphi_{n_1,n_2}(x), x^2 \right\rangle
     -\left\langle D_\xi P_{n_2} D\varphi_{n_1,n_2}(x), (P_{n_2}x)^2 \right\rangle}{1+V}\stackrel{\pi}{=}0
\end{equation}

\noindent
{\bf Step 3}.
By Proposition \ref{p.B.17aa}    for any $n_1,n_2\in \Nset$ there exists a sequence (we assume that it has one index) 
   $\{\varphi_{n_1,n_2,n_3}\}_{n_3\in \Nset} \subset \mathcal E_A(H)$ such that
\begin{equation} \label{e.B.43}
  \lim_{n_3\to \infty }  \frac{\varphi_{n_1,n_2,n_3}}{1+V}
        \stackrel{\pi}{=} \frac{\varphi_{n_1,n_2}}{1+V}
\end{equation}
\begin{equation} \label{e.B.44}
  \lim_{n_3\to \infty }\frac{L_0\varphi_{n_1,n_2,n_3}}{1+V}  
       \stackrel{\pi}{=} \frac{L\varphi_{n_1,n_2}}{1+V}
\end{equation}
and
\[
    \lim_{n_3\to \infty} 
    \frac{ \langle D_\xi   D\varphi_{n_1,n_2,n_3},h\rangle }{1+|\cdot|_{6}^8}
    \stackrel{\pi}{=}
     \frac{\langle D_\xi   D\varphi_{n_1,n_2 },h\rangle}{1+|\cdot|_{6}^8}, \quad \forall h\in H.
\]
Then it follows 
\begin{equation}\label{e.B.45}
    \lim_{n_3\to \infty} \frac{\Big\langle D_\xi P_{n_2} D\varphi_{n_1,n_2,n_3},\left(\cdot\right)^2\Big\rangle}{1+V} 
 \stackrel{\pi}{=}
     \frac{ \left\langle D_\xi P_{n_2} D\varphi_{n_1,n_2},\left(\cdot\right)^2\right\rangle}{1+V}.
\end{equation}

\noindent
{\bf Step 4}.
By construction, $( \varphi_{n_1,n_2,n_3})_{n_1,n_2,n_3}\subset \mathcal E_A(H)$.
By \eqref{e.B.39}, \eqref{e.B.40}, \eqref{e.B.43}
\[
  \lim_{n_1\to \infty}\lim_{n_2\to \infty} \lim_{n_3\to \infty }\frac{\varphi_{n_1,n_2,n_3}}{1+V}
    \stackrel{\pi}{=}\frac{\varphi}{1+V}.
\]
Hence \eqref{e.B.37} follows.
Let us show \eqref{e.B.38}.
By Lemma \ref{l.B.18}, for any $n_1,n_2,n_3\in \Nset$, $x\in L^6(0,1)$ we have 
\[
  K\varphi_{n_1,n_2,n_3}(x)=K_0\varphi_{n_1,n_2,n_3}(x)=
    L_0\varphi_{n_1,n_2,n_3}(x)-\frac12\left\langle D_\xi \varphi_{n_1,n_2,n_3}(x),x^2\right\rangle. 
\]
By \eqref{e.B.44}, \eqref{e.B.45},
\[
 \lim_{n_3\to \infty } \frac{K_0\varphi_{n_1,n_2,n_3} }{1+V}\stackrel{\pi}{=}
                       \frac{  L\varphi_{n_1,n_2}-\frac12 
            \left\langle D_\xi D\varphi_{n_1,n_2},(\cdot)^2\right\rangle   }{1+V}
\]
By \eqref{e.B.41} it holds
\begin{multline*}
   L\varphi_{n_1,n_2}-\frac12  \left\langle D_\xi D\varphi_{n_1,n_2},(\cdot)^2\right\rangle
       \\
  = K_{n_2}D\varphi_{n_1,n_2} +\frac12\left\langle D_\xi P_{n_2} D\varphi_{n_1,n_2}, (P_{n_2}\cdot)^2 \right\rangle
  - \frac12\left\langle D_\xi D\varphi_{n_1,n_2},(\cdot)^2 \right\rangle 
\end{multline*}
By \eqref{e.B.40}, \eqref{e.B.42}
\[
 \lim_{n_3\to \infty }    \frac{      L\varphi_{n_1,n_2,n_3}-\frac12  \left\langle D_\xi D\varphi_{n_1,n_2,n_3},(\cdot)^2\right\rangle    }{1+V}
     \stackrel{\pi}{=}\frac{ K\varphi_{n_1,n_2}}{1+V} 
\]
Finally, by \eqref{e.B.39}, \eqref{e.B.40} we have
\[
  \lim_{n_1\to \infty }\lim_{n_2\to \infty }
    \frac{K\varphi_{n_1,n_2}}{1+V} \stackrel{\pi}{=}  \frac{K\varphi}{1+V}.
\]

\section{The Fokker-Planck equation and Proof of Theorem \ref{T.B.1.4}}

Since the semigroup $P_t$ ,$t\geq0$ is markovian and Theorem \ref{t.B.7.1} holds, it follows that $P_t$ , $t\geq0$ acts on $\mathcal  C_b(L^6(0,1))$.
In particular, we have that $P_t$, $t\geq0$ is a $\pi$-semigroup on $\mathcal  C_b(L^6(0,1))$, as introduced in \cite{Priola}.
We define its infinitesimal generator by
\begin{equation}  \label{e.B.0a}
\begin{cases}
\DS D(K,\mathcal C_b(L^6(0,1)))=\bigg\{ \varphi \in  D(K,\mathcal C_b(L^6(0,1))): \exists g\in \mathcal C_b(L^6(0,1)),   \\ 
  \DS  \qquad  \lim_{t\to 0^+} \frac{ {P}_t\varphi(x)-\varphi(x)}{t}= g(x),\,x\in L^6(0,1),\;
    \sup_{t\in(0,1)}\left\|\frac{ {P}_t\varphi-\varphi}{t}\right\|_0<\infty \bigg\}\\    
   {}   \\
  \DS  {K}\varphi(x)=\lim_{t\to 0^+} \frac{ {P}_t\varphi(x)-\varphi(x)}{t},\quad \varphi\in D(K,\mathcal C_b(L^6(0,1))),\,x\in L^6(0,1).
\end{cases}
\end{equation}

We have the following
\begin{Theorem} \label{T.1} 
The family of  linear maps $ P_t^* :(\mathcal C_b(L^6(0,1)))^*\to (\mathcal C_b(L^6(0,1)))^*$, $t\geq0$, defined by the formula
\begin{equation*}
  \langle \varphi, P_t^*F\rangle_{\sigma(\mathcal C_b(L^6(0,1)),\,(\mathcal C_b(L^6(0,1)))^*)} = 
  \langle P_t\varphi, F\rangle_{\sigma(\mathcal C_b(L^6(0,1)),\, (\mathcal C_b(L^6(0,1)))^*)},\, 
\end{equation*}
where $t\geq0, \, F\in (\mathcal C_b(L^6(0,1)))^*,\, \varphi \in \mathcal C_b(L^6(0,1))$, is a  semigroup of linear operators  on $(\mathcal C_b(L^6(0,1)))^*$ which is stable on $\mathcal M(L^6(0,1))$. 
Moreover,   for any $\mu\in \mathcal M(L^6(0,1))$ there exists a unique family of measures  $\{\mu_t,\;t\geq0\}\subset \mathcal M (L^6(0,1))$
fulfilling
\begin{equation} \label{e.5b}
   \int_0^T|\mu_t|_{TV}(L^6(0,1))dt<\infty,\quad T>0; 
\end{equation} 
\begin{multline}  \label{e.2.6}
 \int_{L^6(0,1)} \varphi(x)\mu_t(dx)-\int_{L^6(0,1)} \varphi(x)\mu (dx)\\
    =\int_0^t\bigg(\int_{L^6(0,1)} K\varphi(x)\mu_s(dx)\bigg)ds,
\end{multline} 
for any $\varphi \in D(K,\mathcal C_b(L^6(0,1)))$, $t\geq0$.
Finally, the solution of \eqref{e.5b}, \eqref{e.2.6} 
is given by $  P_t^*\mu$, ${t\geq0}$.
\end{Theorem}
\begin{proof}
The proof of this theorem is very similar to the proof of  Theorem 1.2 of \cite{Manca07}.
We stress that in  \cite{Manca07} the space $L^6(0,1)$ is replaced by a separable Hilbert space $H$ and that $\mathcal C_b(L^6(0,1))$ is replaced by the space $C_b(H)$, the Banach space of all the uniformly continuous and bounded functions $\varphi:H\to \Rset$, endowed with the supremum norm.
However, one can see that all the results remain true with continuity replacing uniform continuity and $L^6(0,1)$ replacing $H$ (see, also,  \cite{Manca08}). 
\end{proof}

Thanks to Propositions \ref{p.B.2.1} and \ref{p.B.2.9}, the above theorem can be extended to the space  $\mathcal C_{b,V}(L^6(0,1))$.
\begin{Theorem} \label{T.B.1}
Let $( {P}_t)_{t\geq0}$ be the semigroup defined by \eqref{e.B.38a} 
and   let  us consider its infinitesimal generator $(K,D(K,\mathcal C_{b,V}(L^6(0,1))))$   
  given    by  \eqref{e.B.0}.
Then,  the formula 
\[
  \langle \varphi,  {P}_t^*F\rangle_{\sigma(\mathcal C_{b,V}(L^6(0,1)),\,(\mathcal C_{b,V}(L^6(0,1)))^*)} = \langle  {P}_t\varphi, F\rangle_{\sigma(\mathcal C_{b,V}(L^6(0,1)),\, (\mathcal C_{b,V}(L^6(0,1)))^*)}
\] 
defines a semigroup $P_t^*$, $t\geq0$ of linear and continuous operators  on $(\mathcal  C_{b,V}(L^6(0,1))^*$ %
which is stable on $\mathcal M_{V}(L^6(0,1))$. 
Moreover,  for any measure $\mu\in \mathcal M_{V}(L^6(0,1))$ there exists a unique family   
   $\{\mu_t,\;t\geq0\}\subset \mathcal M_{V}(L^6(0,1))$  %
such that \eqref{e.B.5b} holds
and  
\begin{multline}   \label{e.B.8} 
   \int_{L^6(0,1)}\varphi(x)\mu_t(dx)- \int_{L^6(0,1)}\varphi(x)\mu(dx)\\=
   \int_0^t\left( \int_{L^6(0,1)} {K}\varphi(x)\mu_s(dx)   \right)ds  
\end{multline}
 for any $t\geq0$, $\varphi\in D(K,\mathcal C_{b,V}(L^6(0,1)))$.
Finally, the solution of \eqref{e.B.5b}, \eqref{e.B.8}  is given by $  {P}_t^*\mu$, ${t\geq0}$.  
\end{Theorem}
\begin{proof}
Since $P_t$, $t\geq0$ acts on $\mathcal C_{b,V}(L^6(0,1))$ (see Proposition \ref{p.B.2.1}), it follows easily that $P_t^*$ acts on $(\mathcal C_{b,V}(L^6(0,1)))^*$.
Let us show that $P_t^*$ is stable on $\mathcal M_{V}(L^6(0,1))$.
Take $\mu\in\mathcal M_{V}(L^6(0,1))$.
By the linearity of $P_t^*$, it is sufficient to take $\mu$ non negative.
By Theorem \ref{T.1} we have $P_t^*\mu\in\mathcal M(L^6(0,1))$.  
Consider a sequence of functions $(\psi_n)_{n\in \Nset}\subset C_b(L^6(0,1))$ such that $\psi_n(x)\to V(x)$ as $n\to \infty$ and $0\leq \psi(x)\leq V(x)$, for any $x\in L^6(0,1)$.
By Proposition \ref{p.B.2.1} we have that 
\begin{eqnarray*}
   \lim_{n\to\infty}P_t\psi_n(x)&=&\lim_{n\to\infty}\int_{L^6(0,1)}\psi_n(y)\pi_t(x,dy)
   \\
   &=&  \int_{L^6(0,1)}V(y)\pi_t(x,dy)
   =P_tV(x)
\end{eqnarray*}
and
\begin{equation} \label{e.B.63}
 P_t\psi_n(x)=  \int_{L^6(0,1)}\psi_n(y)\pi_t(x,dy) \leq  \int_{L^6(0,1)}V(y)\pi_t(x,dy) \leq c(1+V(x))
\end{equation}
hold for any $x\in L^6(0,1)$ and for some $c>0$.
Hence, since $\mu\in \mathcal M_V(L^6(0,1))$ by the dominated convergence Theorem we have
\begin{eqnarray*}
  \int_{L^6(0,1)} V(x)P_t^*\mu(dx)&=&\lim_{n\to\infty} \int_{L^6(0,1)} \psi_n(x) P_t^*\mu(dx)\\
  &=& \lim_{n\to\infty}  \int_{L^6(0,1)}P_t \psi_n(x) \mu(dx)\\
 &=&\lim_{n\to\infty} \int_{L^6(0,1)} \left(\int_{L^6(0,1)} \psi_n (y)\pi_t(x,dy) \right)\mu(dx) \\
 &=&  \int_{L^6(0,1)} P_tV(x)\mu(dx). 
\end{eqnarray*}
Then, by \eqref{e.B.63}
\[
   \int_{L^6(0,1)} V(x)P_t^*\mu(dx)\leq \int_{L^6(0,1)} \leq c(1+V(x))\mu(dx)<\infty,
\]
since $\mu \in\mathcal M_V(L^6(0,1))$.
Recalling that $P_t^*\mu\in \mathcal M(L^6(0,1))$, it follows  $P_t^*\mu\in\mathcal M_V(L^6(0,1))$.

Let us now prove the remaining part of the theorem.
We split the proof in two steps.\\
{\bf Step 1: Existence of a solution of \eqref{e.B.5b}, \eqref{e.B.8}}.
Let us fix $\mu\in \mathcal M_V(L^6(0,1))$.
By the first part of the Theorem, we have $P_s^*\mu\in \mathcal M_V(L^6(0,1))$, for any $s\geq0$.
We now show that for any $\varphi\in D(K, \mathcal C_{b,V}(L^6(0,1)))$
the function
\[
  [0,\infty)\to\Rset,\quad s\mapsto \int_{L^6(0,1)}P_s\varphi(x)\mu(dx) 
\]
is differentiable, with continuous differential 
\[
  t\mapsto \int_{L^6(0,1)}K\varphi(x)P_s^*\mu(dx). 
\]
By \eqref{e.B.0} we have, for any $\varphi\in D(K, \mathcal C_{b,V}(L^6(0,1)))$, 
\begin{eqnarray*}
  && \lim_{h\to 0^+} \frac{1}{h}\left(\int_{L^6(0,1)}P_{s+h}\varphi(x)\mu(dx) -\int_{L^6(0,1)}P_s\varphi(x)\mu(dx)\right)
    \notag
\\
 &&\qquad = \lim_{h\to 0^+} \frac{1}{h}\left(\int_{L^6(0,1)}KP_h\varphi(x)P_s^*\mu(dx) -\int_{L^6(0,1)} \varphi(x)P_s^*\mu(dx)\right)\notag
\\
 &&\qquad = \lim_{h\to 0^+} \int_{L^6(0,1)}\frac{P_h\varphi(x)-\varphi(x)}{h}P_s^*\mu(dx)\notag
\\
 &&\qquad  = \int_{L^6(0,1)}K\varphi(x)P_s^*\mu(dx) \notag
\end{eqnarray*}
Since $K\varphi\in \mathcal C_{b,V}(L^6(0,1))$, we have
\[
  \int_{L^6(0,1)}K\varphi(x)P_s^*\mu(dx)= \int_{L^6(0,1)}P_sK\varphi(x)\mu(dx).
\]
Then, by Proposition \ref{p.B.2.9}, this is a continuous as function of $s$.
By the fundamental theorem of calculus it follows that $P_t^*\mu$, $t\geq0$ solves 
\eqref{e.B.5b}, \eqref{e.B.8}.\\
{\bf Uniqueness of a solution}.
Assume that $\{\mu_t,\,t\geq0\}\subset \mathcal M_V(L^6(0,1))$ fulfills \eqref{e.B.5b}, \eqref{e.B.8}. 
It is straightforward to show that $D(K,C_b(L^6(0,1)))\subset D(K,C_{b,V}(L^6(0,1)))$.
Then \eqref{e.B.8} holds for any $\varphi\in D(K,C_b(L^6(0,1)))$.
It is also obvious that  \eqref{e.B.5b} implies \eqref{e.5b}.
Then, by Theorem \ref{T.1}, if follows that $\mu_t=P_t^*\mu$, for any $t\geq0$.   
\end{proof}

\subsection{Proof of Theorem \ref{T.B.1.4}}

Take $\mu\in \mathcal M_V(L^6(0,1))$. 
As before, we split the proof in two steps.

\noindent
{\bf Existence of a solution of \eqref{e.B.5b}, \eqref{e.B.2.5a}.}
Fix $\mu\in \mathcal M_V(L^6(0,1))$.
By Theorem \ref{T.B.1} we have that the family $P_t^*\mu$, $t\geq0 $ fulfills \eqref{e.B.5b}.
On the other hand, by Theorem \ref{T.B.2} we have $\mathcal E_A(H)\subset D(K, \mathcal C_{b,V}(L^6(0,1)))$
and $K\varphi=K_0\varphi$ for any $\varphi\in \mathcal E_A(H)$.
Then, by \eqref{e.B.8} it follows that \eqref{e.B.2.5a} holds for any $t\geq0$, $\varphi\in \mathcal E_A(H)$.
Then, $P_t^*\mu$, $t\geq0 $ is solution of \eqref{e.B.5b}, \eqref{e.B.2.5a}.

\noindent
{\bf Uniqueness of the solution.}
Assume that $\{\mu_t,\;t\geq0\}\subset \mathcal M_V(L^6(0,1))$ fulfills \eqref{e.B.5b}, \eqref{e.B.2.5a}.
Take $\varphi\in \mathcal C_{b,V}(L^6(0,1))$.
By Theorem \ref{T.B.2}  there exist $m\in\Nset$ and   an $m$-indexed sequence 
 $(\varphi_{n_1,\ldots,n_m})_{n_1,\ldots,n_m\in\Nset}\subset \mathcal E_A(H)$ such that
\begin{equation*} 
   \lim_{n_1\to \infty}\cdots\lim_{n_m\to \infty}  \frac{ \varphi_{n_1,\ldots,n_m}}{1+V} \stackrel{\pi}{=}\frac{ \varphi}{1+V} 
\end{equation*}
and
 \begin{equation*}
\lim_{n_1\to \infty}\cdots\lim_{n_m\to \infty}  \frac{ K_0\varphi_{n_1,\ldots,n_m}}{1+V} \stackrel{\pi}{=}\frac{ K\varphi}{1+V} 
\end{equation*}
Then, since $\{\mu_t,\;t\geq0\} \subset \mathcal M_V(L^6(0,1))$, by the dominated convergence theorem we have
\[
  \lim_{n_1\to \infty}\cdots\lim_{n_m\to \infty} \left( \int_{L^6(0,1)}\varphi_{n_1,\ldots,n_m}(x)\mu_t(dx)- \int_{L^6(0,1)}\varphi_{n_1,\ldots,n_m}(x)\mu(dx)     \right)
\]
\[
 =  \int_{L^6(0,1)}\varphi(x)\mu_t(dx)- \int_{L^6(0,1)}\varphi(x)\mu(dx)    
\]
Similarly, for any $s\in [0,t]$ we have
\[
 \lim_{n_1\to \infty}\cdots\lim_{n_m\to \infty}\int_{L^6(0,1)}K_0\varphi_{n_1,\ldots,n_m}(x)\mu_s(dx)
\]
\[
   =\int_{L^6(0,1)}K \varphi (x)\mu_s(dx).
\]
Therefore, by  \eqref{e.B.5b} we can still apply the dominated convergence theorem to find
\[
   \lim_{n_1\to \infty}\cdots\lim_{n_m\to \infty}\int_0^t\left(\int_{L^6(0,1)}K_0\varphi_{n_1,\ldots,n_m}(x)\mu_s(dx)\right)ds
\]
\[
   =\int_0^t\left(\int_{L^6(0,1)}K \varphi (x)\mu_s(dx)\right)ds.
\] 
Then, $\{\mu_t,\;t\geq0\} $ is solution of \eqref{e.B.5b} and  \eqref{e.B.8}, 
for any $\varphi\in D(K,C_{b,V}(L^6(0,1)))$.
But by Theorem \ref{T.B.1} such a solution is unique, thus $ \mu_t  $ must coincide with 
$  P_t^*\mu   $, $\forall t\geq0$.
The proof is complete. 
\qed

\end{document}